\documentclass[oneside,american,english]{amsart}
\usepackage[T1]{fontenc}
\usepackage[latin9]{inputenc}
\usepackage{amsthm}
\usepackage{amstext}
\usepackage{amssymb}
\usepackage{esint}

\makeatletter
\numberwithin{equation}{section}
\numberwithin{figure}{section}
  \theoremstyle{plain}
  \newtheorem*{thm*}{\protect\theoremname}
\theoremstyle{plain}
\newtheorem{thm}{\protect\theoremname}
  \theoremstyle{definition}
  \newtheorem{defn}[thm]{\protect\definitionname}
  \theoremstyle{plain}
  \newtheorem{prop}[thm]{\protect\propositionname}
  \theoremstyle{definition}
  \newtheorem{example}[thm]{\protect\examplename}
  \theoremstyle{plain}
  \newtheorem{cor}[thm]{\protect\corollaryname}
  \theoremstyle{remark}
  \newtheorem{rem}[thm]{\protect\remarkname}
  \theoremstyle{plain}
  \newtheorem{lem}[thm]{\protect\lemmaname}

\usepackage{ae,aecompl} 

\usepackage{enumitem}

\renewenvironment{enumerate}{\begin{oldenumerate}[topsep=0pt]}{\end{oldenumerate}}

\makeatletter
\newtheorem*{rep@theorem}{\rep@title}
\newcommand{\newreptheorem}[2]{%
\newenvironment{rep#1}[1]{%
 \def\rep@title{#2 \ref{##1}}%
 \begin{rep@theorem}}%
 {\end{rep@theorem}}}
\makeatother

\newreptheorem{theorem}{Theorem}

\makeatother

\usepackage{babel}
  \addto\captionsamerican{\renewcommand{\corollaryname}{Corollary}}
  \addto\captionsamerican{\renewcommand{\definitionname}{Definition}}
  \addto\captionsamerican{\renewcommand{\examplename}{Example}}
  \addto\captionsamerican{\renewcommand{\lemmaname}{Lemma}}
  \addto\captionsamerican{\renewcommand{\propositionname}{Proposition}}
  \addto\captionsamerican{\renewcommand{\remarkname}{Remark}}
  \addto\captionsamerican{\renewcommand{\theoremname}{Theorem}}
  \addto\captionsenglish{\renewcommand{\corollaryname}{Corollary}}
  \addto\captionsenglish{\renewcommand{\definitionname}{Definition}}
  \addto\captionsenglish{\renewcommand{\examplename}{Example}}
  \addto\captionsenglish{\renewcommand{\lemmaname}{Lemma}}
  \addto\captionsenglish{\renewcommand{\propositionname}{Proposition}}
  \addto\captionsenglish{\renewcommand{\remarkname}{Remark}}
  \addto\captionsenglish{\renewcommand{\theoremname}{Theorem}}
  \providecommand{\corollaryname}{Corollary}
  \providecommand{\definitionname}{Definition}
  \providecommand{\examplename}{Example}
  \providecommand{\lemmaname}{Lemma}
  \providecommand{\propositionname}{Proposition}
  \providecommand{\remarkname}{Remark}
  \providecommand{\theoremname}{Theorem}
\providecommand{\theoremname}{Theorem}

\begin{document}
\selectlanguage{american}%
\global\long\def\epsilon{\varepsilon}

\global\long\def\phi{\varphi}

\global\long\def\R{\mathbb{R}}

\global\long\def\N{\mathbb{N}}

\global\long\def\o{\mathbf{1}}

\global\long\def\T{\mathcal{T}}

\global\long\def\S{\mathcal{S}}

\global\long\def\L{\mathcal{L}}

\global\long\def\kn{\mathcal{K}^{n}}

\global\long\def\vol{\text{Vol}}

\global\long\def\lc{\operatorname{LC}\left(\R^{n}\right)}

\global\long\def\lco{\operatorname{LC}_{0}\left(\R^{n}\right)}

\global\long\def\qc{\operatorname{QC}\left(\R^{n}\right)}

\global\long\def\qco{\operatorname{QC}_{0}\left(\R^{n}\right)}

\global\long\def\cvx{\operatorname{Cvx}\left(\R^{n}\right)}

\global\long\def\KL{\underline{K}}

\global\long\def\KU{\overline{K}}

\selectlanguage{english}%

\title{Mixed integrals and related inequalities}

\author{Vitali Milman}

\address{School of Mathematical Sciences \\
Tel Aviv University \\
Tel Aviv 69978, Israel}

\email{milman@post.tau.ac.il}

\thanks{Partially supported by the Minkowski Center at the University of
Tel Aviv, by ISF grant 387/09 and by BSF grant 2006079.}

\author{Liran Rotem}

\address{School of Mathematical Sciences \\
Tel Aviv University \\
Tel Aviv 69978, Israel}

\email{liranro1@post.tau.ac.il}
\begin{abstract}
In this paper we define an addition operation on the class of quasi-concave
functions. While the new operation is similar to the well-known sup-convolution,
it has the property that it polarizes the Lebesgue integral. This
allows us to define mixed integrals, which are the functional analogs
of the classic mixed volumes.

We extend various classic inequalities, such as the Brunn-Minkowski
and the Alexandrov-Fenchel inequality, to the functional setting.
For general quasi-concave functions, this is done by restating those
results in the language of rearrangement inequalities. Restricting
ourselves to log-concave functions, we prove generalizations of the
Alexandrov inequalities in a more familiar form. 
\end{abstract}



\maketitle

\section{Introduction}

One of the fundamental theorems in classic convexity is Minkowski's
theorem on mixed volumes. In order to state the theorem, we will need
some basic definitions. Denote by $\kn$ the class of all closed,
convex sets in $\R^{n}$. On $\kn$ we have the operation of Minkowski
addition, defined by 
\[
K_{1}+K_{2}=\left\{ x_{1}+x_{2}:\ x_{1}\in K_{1},\ x_{2}\in K_{2}\right\} .
\]
 Similarly, if $K\in\kn$ and $\lambda\ge0$, we can define the homothet
$\lambda\cdot K$ as 
\[
\lambda\cdot K=\left\{ \lambda x:\ x\in K\right\} .
\]
 Finally, for $K\in\kn$ define $\vol(K)\in[0,\infty]$ to be the
standard Lebesgue volume of $K$. Now we can state Minkowski's theorem
(see, e.g. \cite{schneider_convex_1993} for a proof):
\begin{thm*}
[Minkowski]Fix $K_{1},K_{2},\ldots,K_{m}\in\kn$. Then the function
$F:\left(\R^{+}\right)^{m}\to[0,\infty]$, defined by 
\[
F(\epsilon_{1},\epsilon_{2},\ldots,\epsilon_{m})=\vol\left(\epsilon_{1}K_{1}+\epsilon_{2}K_{2}+\cdots+\epsilon_{m}K_{m}\right),
\]
 is a homogenous polynomial of degree $n$, with positive coefficients.
\end{thm*}
The main goal of this paper is to extend Minkowski's theorem and related
inequalities, from the class of convex bodies to the larger classes
of log-concave and quasi-concave functions. We will soon give the
relevant definitions and the exact statements, but first let us make
a few comments about Minkowski's theorem. 

First, notice that we did not make the usual assumption that the sets
$K_{i}$ are compact. This is not a problem, as long as we allow our
polynomial to attain the value $+\infty$ and adopt the convention
that $0\cdot\infty=0$. Second, by standard linear algebra, Minkowski's
theorem is equivalent to the existence of a polarization for the volume
form. More explicitly, there exists a function 
\[
V:\left(\kn\right)^{n}\to[0,\infty]
\]
 which is multilinear, symmetric, and satisfies $V(K,K,\ldots,K)=\vol(K)$.
The number $V(K_{1},K_{2},\ldots,K_{n})$ is called the mixed volume
of the $K_{1},K_{2},\ldots,K_{n}$, and is nothing more than the relevant
coefficient of the Minkowski polynomial:
\[
\vol\left(\epsilon_{1}K_{1}+\epsilon_{2}K_{2}+\cdots+\epsilon_{m}K_{m}\right)=\sum_{i_{1},i_{2},\ldots,i_{n}=1}^{m}\epsilon_{i_{1}}\epsilon_{i_{2}}\cdots\epsilon_{i_{n}}\cdot V(K_{i_{1}},K_{i_{2}},\ldots,K_{i_{n}}).
\]

We would also like to note that more than anything, Minkowski's theorem
is a property of the Minkowski addition. To put this comment in perspective,
notice that there are several interesting ways to define the sum of
convex bodies, and the Minkowski addition is just one of the possibilities.
For example, remember that the support function of a convex body $K\in\kn$
is a 1-homogenous convex function $h_{K}:\R^{n}\to\R\cup\left\{ +\infty\right\} $
defined by 
\[
h_{K}(y)=\sup_{x\in K}\left\langle x,y\right\rangle .
\]
 Support functions are connected to the Minkowski sum via the relation
\[
h_{K+T}(x)=h_{K}(x)+h_{T}(x).
\]
 Similarly, if $K$ and $T$ are convex sets containing the origin,
then for every $1\le p\le\infty$ we can define their $L_{p}$-sum
$K+_{p}T$ using the relation 
\[
h_{K+_{p}T}(x)=\left(h_{K}(x)^{p}+h_{T}(x)^{p}\right)^{\frac{1}{p}}.
\]

As a second example, remember that if $K$ is a convex body containing
the origin, then its polar body is another convex body defined by
\[
K^{\circ}=\left\{ y\in\R^{n}:\ h_{K}(y)\le1\right\} .
\]
 For two such bodies $K$ and $T$, we can define a summation operation
by $K\oplus T=\left(K^{\circ}+T^{\circ}\right)^{\circ}$.

In both of the above examples we only defined the addition, and not
the homothety operation. However, it is a general fact that sufficiently
``nice'' addition operations on $\kn$ induce a natural homothety
operation. Specifically, for $m\in\N$ one can always define 
\[
m\cdot K=\underbrace{K+K+\cdots+K}_{m\text{ times}}.
\]
It is often the case that for every $K\in\kn$ and $m\in\N$ there
exists a unique body $T\in\kn$ such that $m\cdot T=K$, and then
it is natural to define $\frac{\ell}{m}\cdot K=\ell\cdot T$. Finally,
one extends the definition to a general $\lambda>0$ using some sort
of continuity. Because of this construction we will suppress the role
of the homotheties in informal discussions, and will talk only about
the addition operation. In other words, we adopt the convention that
homotheties are always the induced from, and compatible with, the
addition operation. It is easy to see that for $L_{p}$-sums the induced
homothety operation is $\lambda\cdot_{p}K=\lambda^{\frac{1}{p}}\cdot K$,
and for the polar sum the induced homothety is $\lambda\odot K=\lambda^{-1}\cdot K$.

Our examples of addition share some appealing properties of the Minkowski
addition. For example, they are all commutative, associative, and
with $\left\{ 0\right\} $ serving as an identity element. However,
using some simple examples, one may check that volume is no longer
a polynomial, if one replaces Minkowski addition by $L_{p}$-addition
(for $p>1$) or polar addition, and the same would be true for ``most''
possible definitions of addition. In fact, in \cite{gardner_operations_2012}
the authors consider the problem of characterizing the Minkowski addition.
Roughly speaking, they show that the Minkowski addition is the only
operation on convex bodies satisfying a short list of properties,
one of which is polynomiality of volume. However, the convention in
\cite{gardner_operations_2012} is that homotheties are always the
classic Minkowski homotheties, and not the ones which are induced
from the addition operation (an interesting related question is whether
there are any addition operations on convex bodies, other than Minkowski
addition, such that the induced homothety operation is the classic
one).

In recent years, it became apparent that embedding the class $\kn$
of convex sets into some class of functions $f:\R^{n}\to[-\infty,\infty]$
can lead to important implications (see the survey \cite{milman_geometrization_2008}).
One natural choice is to embed $\kn$ into the class of convex functions,
by mapping each $K\in\kn$ to its convex indicator function, defined
as 
\[
\o_{K}^{\infty}(x)=\begin{cases}
0 & x\in K\\
\infty & \text{otherwise. }
\end{cases}
\]
 This embedding, however, has the technical disadvantage that convex
functions are almost never integrable. To remedy the situation, we
usually deal with log-concave functions, which are functions $f:\R^{n}\to[0,\infty)$
of the form $f=e^{-\phi}$, where $\phi$ is a convex function. In
particular, every convex set $K$ is mapped to its standard indicator
function, 
\[
\o_{K}(x)=\begin{cases}
1 & x\in K\\
0 & \text{otherwise}.
\end{cases}
\]

To be a bit more formal, we define 
\[
\cvx=\left\{ \phi:\R^{n}\to(-\infty,\infty]:\ \phi\text{ is convex and lower semicontinuous }\right\} ,
\]
 and then 
\[
\lc=\left\{ e^{-\phi}:\ \phi\in\cvx\right\} 
\]
 is the class of log-concave functions. The semi-continuity assumption
is just the analog of the assumption that our convex sets are closed. 

We would like to find a Minkowski-type theorem for log-concave functions.
In order to achieve this goal, we first need to give meaning to the
concept of ``volume'', and the concept of ``addition''. For volume,
we need some functional $I:\lc\to[0,\infty)$ such that $I(\o_{K})=\vol(K)$
for all convex bodies $K$. The obvious candidate is the Lebesgue
integral, 
\[
I(f)=\int_{\R^{n}}f(x)dx.
\]

For addition, matters are more complicated. Defining addition on $\lc$
is, of course, equivalent to defining addition on $\cvx$ -- for any
operation $\oplus$ on $\cvx$ we can define an operation on $\lc$
by $e^{-\phi}\oplus e^{-\psi}=e^{-\left(\phi\oplus\psi\right)}$.
The first attempt at a definition is probably the pointwise addition,
\[
\left(\phi+\psi\right)(x)=\phi(x)+\psi(x),
\]
 which transforms to pointwise multiplication for log-concave functions.
This definition, however, has many problems, not the least of which
is that it does not extend Minkowski addition:
\[
\o_{K}(x)\cdot\o_{T}(x)=\o_{K\cap T}(x)\ne\o_{K+T}(x).
\]

A better definition, and the one that is usually used in applications,
is that of inf-convolution: 
\[
\left(\phi\square\psi\right)(x)=\inf_{y\in\R^{n}}\left[\phi(y)+\psi(x-y)\right].
\]
 The corresponding operation for log-concave functions is the so called
sup-convolution or Asplund sum, defined by 
\[
\left(f\star g\right)(x)=\sup_{y\in\R^{n}}f(y)g(x-y).
\]
 The sup-convolution generalized the Minkowski addition, in the sense
that $\o_{K}\star\o_{T}=\o_{K+T}$. However, there is no Minkowski
type theorem for this operation. This is easy to see, as the Lebesgue
integral is not even homogenous with respect to the sup-convolution:
For a general $f\in\lc$ we do not have 
\[
\int\left(f\star f\right)=2^{n}\int f,
\]
 as one verifies with simple examples.

We will now define another operation on convex functions (and, by
extension, on log-concave functions as well):
\begin{defn}
\label{def:quasi-sum}The sum of convex functions $\phi,\psi\in\cvx$
is 
\[
\left(\phi\oplus\psi\right)(x)=\inf_{y\in\R^{n}}\max\left\{ \phi(y),\psi(x-y)\right\} .
\]
Additionally, if $\lambda>0$ we define the product $\lambda\odot\phi$
as 
\[
\left(\lambda\odot\phi\right)(x)=\phi\left(\frac{x}{\lambda}\right).
\]

\end{defn}
On the level of log-concave functions, the operation $\oplus$ is
defined by 
\[
\left(f\oplus g\right)(x)=\sup_{y\in\R^{n}}\min\left\{ f(y),g(x-y)\right\} .
\]

For Definition \ref{def:quasi-sum} to make sense, we need to know
that $\phi\oplus\psi$ is convex whenever $\phi$ and $\psi$ are.
We will prove this, together with other properties of $\oplus$, in
section \ref{sec:Minkwoski-theorem}. For now, let us highlight the
main features of this operation.

First, one easily checks that $\odot$ really is the homothety operation
induced from $\oplus$. In particular, we have $2\odot\phi=\phi\oplus\phi$
for every $\phi\in\cvx$. Second, the operation $\oplus$ extends
the Minkowski addition on convex bodies, in the sense that 
\[
\o_{K}\oplus\o_{T}=\o_{K+T}.
\]
Third, the operation $\oplus$ is not so different from the more classic
inf-convolution $\square$. In fact, it follows from Proposition \ref{prop:old-new-relation}
that if $\phi$ and $\psi$ are positive, convex functions, then 
\[
\frac{1}{2}\left(\phi\square\psi\right)(x)\le\left(\phi\oplus\psi\right)(x)\le\left(\phi\square\psi\right)(x),
\]
 i.e. both operations agree up to a factor of 2. Furthermore, if $\phi$
is a positive, convex function and $K\in\kn$ is any convex set, then
\[
\left(\phi\oplus\o_{K}^{\infty}\right)(x)=\left(\phi\square\o_{K}^{\infty}\right)(x)=\inf_{y\in K}\phi(x-y),
\]
 so in this case both operations are exactly the same. 

However, our new addition has one critical advantage over the better
known inf-convolution: The volume functional $I(f)$ polarizes with
respect to $\oplus$. In other words, we have the following Minkowski-type
theorem:
\begin{thm*}
Fix $f_{1},f_{2},\ldots,f_{m}\in\lc$. Then the function $F:\left(\R^{+}\right)^{m}\to[0,\infty]$,
defined by 
\[
F(\epsilon_{1},\epsilon_{2},\ldots,\epsilon_{m})=\int\left[\left(\epsilon_{1}\odot f_{1}\right)\oplus\left(\epsilon_{2}\odot f_{2}\right)\oplus\cdots\oplus\left(\epsilon_{m}\odot f_{m}\right)\right]
\]
 is a homogenous polynomial of degree $n$, with non-negative coefficients.
\end{thm*}
In complete analogy with the case of convex bodies, this theorem is
equivalent to the existence of a function 
\[
V:\lc^{n}\to[0,\infty]
\]
 which is symmetric, multilinear (with respect to $\oplus$, of course)
and satisfies $V(f,f,\ldots,f)=\int_{\R^{n}}f(x)dx$. We will call
the number $V(f_{1},f_{2},\ldots f_{n})$ the mixed integral of $f_{1},f_{2},\ldots,f_{n}$. 

When reading the proof of our Minkowski-type theorem, one can see
that log-concavity is never used in any real way. In fact, everything
we said until this point will remain true, if log-concave functions
are replaced with the more general quasi-concave functions:
\begin{defn}
A function $f:\R^{n}\to\R$ is called quasi-concave if 
\[
f(\lambda x+(1-\lambda)y)\ge\min\left\{ f(x),f(y)\right\} 
\]
 for every $x,y\in\R^{n}$ and $0<\lambda<1$. The class of all functions
$f:\R^{n}\to[0,\infty)$ which are upper semicontinuous and quasi-concave
functions will be denoted by $\qc$. 

Similarly, a function $\phi:\R^{n}\to\R$ is called quasi-convex if
\[
f\left(\lambda x+(1-\lambda)y\right)\le\max\left\{ f(x),f(y)\right\} .
\]
 We will not have a special notation for the class of quasi-convex
functions.
\end{defn}
Quasi-concave functions are frequently used by economists (see, e.g.,
\cite{simon_mathematics_1994}). One of the main reasons for this
is that quasi-concavity is an ``ordinal property''. Let us explain
this point: given a function $f:\R^{n}\to[0,\infty)$ and an increasing
function $\rho:[0,\infty)\to[0,\infty)$, we will say that the function
$\rho\circ f$ is a rescaling of $f$. Many important functions in
economy (e.g. the utility function) are ordinal, that is defined only
up to rescaling. Remember that even if a function $f$ is concave,
its rescaling $\rho\circ f$ need not be concave. Hence one cannot
talk, for example, about ``concave utility functions'' -- concavity
is not an ordinal property. In contrast, it is easy to check that
if $f$ is quasi-concave, every rescaling of it will be quasi-concave
as well. 

As far as we know, quasi-concave functions were never a serious object
of study from the convex geometry point of view. One of the main points
of this paper is to show that the realm of quasi-concave functions
is the natural setting for many results and theorems.

After we define mixed integrals in section \ref{sec:Minkwoski-theorem},
section \ref{sec:Inequalities} is devoted to proving many different
inequalities between these numbers. A sizable portion of classic convexity
theory involves proving inequalities between different mixed volumes.
For example, if $K\in\kn$ is a convex body, its surface area is defined
to be 
\[
S(K)=\lim_{\epsilon\to0^{+}}\frac{\vol\left(K+\epsilon D\right)-\vol\left(K\right)}{\epsilon},
\]
 where $D$ is the unit Euclidean ball. It is not hard to see that
$S(K)$ is a mixed volume. In fact, we have 
\[
S(K)=n\cdot V(\underbrace{K,K,\ldots,K}_{n-1\text{ times}},D).
\]
 The famous isoperimetric inequality states that out of all bodies
with fixed volume, the Euclidean ball has the minimal surface area.
More quantitatively, it is usually written as 
\[
S(K)\ge n\cdot\vol(D)^{\frac{1}{n}}\cdot\vol\left(K\right)^{\frac{n-1}{n}}.
\]
 A proof of the isoperimetric inequality, as well as all the other
inequalities of mixed volumes which appear in this paper, can be found
in \cite{schneider_convex_1993}.

In section \ref{sec:Inequalities} we discuss the question of how
to prove a generalization of this theorem to quasi-concave functions.
We define the surface area of a quasi-concave function $f$ to be
\[
S(f)=n\cdot V(\underbrace{f,f,\ldots,f}_{n-1\text{ times}},\o_{D}).
\]
For log-concave functions the surface area $S(f)$ was discovered
independently by Colesanti (\cite{colesanti_what_2012}). We will
give more details about his work in Example \ref{example:quermass}.

Naively, we may try and bound $S(f)$ from below in terms of the integral
$\int f$. Unfortunately, we will see that no such bound can exist
for arbitrary quasi-concave functions.

Instead, we will employ another approach. For every $K\in\kn$ let
$K^{\ast}$ be the Euclidean ball with the same volume as $K$. Then
the isoperimetric inequality can be stated as $S(K)\ge S(K^{\ast})$
for every $K\in\kn.$ Similarly, if $f\in\qc$, we define $f^{\ast}\in\qc$
to be its symmetric decreasing rearrangement (see Definition \ref{def:SDR},
and see \cite{lieb_analysis_1997} for more information. For the purpose
of this introduction we will assume that $f$ is ``nice'' enough
for $f^{\ast}$ to be well-defined). We will prove the following isoperimetric
inequality:
\begin{thm*}
For every $f\in\qc$ we have $S(f)\ge S(f^{\ast}),$ with equality
if and only if $f$ is rotation invariant.
\end{thm*}
This inequality generalizes the classic isoperimetric inequality.
It can also be useful for general quasi-concave functions, because
it reduces an $n$-dimensional problem to a 1-dimensional one -- the
function $f^{\ast}$ is rotation invariant, and hence essentially
``one dimensional''. However, we stress again that in general, this
inequality does not yield a lower bound for $S(f)$ in terms of $\int f$,
as such a bound is impossible.

In section \ref{sec:Inequalities} we generalize many important inequalities
by rewriting them as rearrangement inequalities. In particular, we
extend both the Brunn-Minkowski theorem (even in its general form
- see Theorem \ref{prop:gen-BM}) and the Alexandrov-Fenchel inequality
(Theorem \ref{prop:AF}). As the statements are rather involved, we
will not reproduce them here. Instead, we will present an elegant
corollary of Theorem \ref{prop:AF}:
\begin{thm*}
For every $f_{1},f_{2},\ldots,f_{n}\in\qc$ we have 
\[
V(f_{1},f_{2},\ldots,f_{n})\ge V(f_{1}^{\ast},f_{2}^{\ast},\ldots f_{n}^{\ast}).
\]
 
\end{thm*}
In section \ref{sec:LC-inequalities}, we once again restrict our
attention to log-concave functions. We demonstrate how one can use
the results of section \ref{sec:Inequalities}, together with a 1-dimensional
analysis, to prove sharp numeric inequalities between mixed integrals.
More specifically, we prove the following Alexandrov type inequality:
\begin{thm*}
Define $g(x)=e^{-\left|x\right|}$. For every $f\in\lc$ with $\max f=1$
and every integers $0\le k<m<n$, we have 
\[
\left(\frac{W_{k}(f)}{W_{k}(g)}\right)^{\frac{1}{n-k}}\le\left(\frac{W_{m}(f)}{W_{m}(g)}\right)^{\frac{1}{n-m}}
\]
 with equality if and only if $f(x)=e^{-c\left|x\right|}$ for some
$c>0$.
\end{thm*}
Here the numbers $W_{k}(f)$ are the quermassintegrals of $f$, defined
by 
\[
W_{k}(f)=V(\underbrace{f,f,\ldots f}_{n-k\text{ times}},\underbrace{\o_{D},\o_{D},\ldots,\o_{D}}_{k\text{ times}})
\]
 (see Example \ref{example:quermass}). In particular, the case $k=0$,
$m=1$ gives a sharp isoperimetric inequality:
\begin{thm*}
For every $f\in\lc$ with $\max f=1$ we have
\[
S(f)\ge\left(\int f\right)^{\frac{n-1}{n}}\cdot\frac{S(g)}{\left(\int g\right)^{\frac{n-1}{n}}}
\]
 with equality if and only if $f(x)=e^{-c\left|x\right|}$ for some
$c>0$. 
\end{thm*}
Finally, in section \ref{sec:Rescaling}, we revisit the notion of
rescaling. Consider for example the Brunn-Minkowski inequality: it
states that for every $A,B\in\kn$ we have 
\[
\vol\left(A+B\right)^{\frac{1}{n}}\ge\vol(A)^{\frac{1}{n}}+\vol(B)^{\frac{1}{n}}.
\]
 The most obvious way to generalize this inequality to the realm of
quasi-concave functions is to ask whether 
\[
\left[\int\left(f\oplus g\right)\right]^{\frac{1}{n}}\ge\left[\int f\right]^{\frac{1}{n}}+\left[\int g\right]^{\frac{1}{n}}
\]
 for arbitrary functions $f,g\in\qc$. Unfortunately, this is false,
just like the most naive way to generalize the isoperimetric inequality
turned out to be false. One way we already discussed to remedy the
situation is to reinterpret the Brunn-Minkowski inequality as a rearrangement
inequality. The resulting inequality will read 
\[
\left(f\oplus g\right)^{\ast}\ge f^{\ast}\oplus g^{\ast}
\]
 (see Proposition \ref{prop:BM} for an exact statement and a detailed
discussion). However, there is also a second way. Remember that quasi-concave
functions are often ordinal functions, i.e. functions defined only
up to rescaling. In such a case, we can ask a more delicate question:
Is it possible to choose rescalings of $f$ and $g$ in such a way
that the Brunn-Minkowski inequality will hold? Often, the answer is
``yes'':
\begin{thm*}
Assume $f$ and $g$ are ``sufficiently nice'' quasi-concave functions.
Then one can rescale $f$ to a function $\tilde{f}$ in such a way
that 
\[
\left[\int\left(\tilde{f}\oplus g\right)\right]^{\frac{1}{n}}\ge\left[\int\tilde{f}\right]^{\frac{1}{n}}+\left[\int g\right]^{\frac{1}{n}}.
\]
 
\end{thm*}
Of course, the exact definition of ``sufficiently nice'' will be
given in section \ref{sec:Rescaling}. Section \ref{sec:Rescaling}
also contains a ``rescaled'' version of the Alexandrov-Fenchel inequality,
which we will not produce here.

\section{\label{sec:Minkwoski-theorem}Minkowski theorem for quasi-concave
functions}

The main goal of this section is to establish the various properties
of the addition $\oplus$ from Definition \ref{def:quasi-sum}, including
a Minkowski type theorem. Remember that the sum $f\oplus g$ of two
quasi-concave functions is defined by 
\[
\left(f\oplus g\right)(x)=\sup_{y\in\R^{n}}\min\left\{ f(y),g(x-y)\right\} .
\]

The first thing we do is give an alternative, more intuitive definition
for $\oplus$. We start by defining
\begin{defn}
\label{def:level-sets}For a quasi-convex function $\phi:\R^{n}\to\R$
and $t\in\R$, we define 
\[
\KL_{t}(\phi)=\left\{ x\in\R^{n}:\ \phi(x)\le t\right\} .
\]
 Similarly for $f\in\qc$ and $t>0$, we define 
\[
\KU_{t}(f)=\left\{ x\in\R^{n}:\ f(x)\ge t\right\} .
\]
 
\end{defn}
Since $\phi$ is quasi-convex, its lower level sets $\KL_{t}(\phi)$
are convex. If it also lower semicontinuous, the sets $\KL_{t}(\phi)$
are closed as well. Therefore $\KL_{t}(\phi)\in\kn$, and for the
similar reasons $\KU_{t}(f)\in\kn$ for $f\in\qc$. We can now explain
our addition in terms of level sets:
\begin{prop}
\label{prop:sum-by-levels}Assume $f,g\in\qc$ and the level sets
$\KU_{t}(f)$ are compact for all $t>0$. Then we have 
\[
\KU_{t}(f\oplus g)=\KU_{t}(f)+\KU_{t}(g).
\]
 Similarly, if $\lambda>0$, and with no compactness assumption, we
get 
\[
\KU_{t}(\lambda\odot f)=\lambda\cdot\KU_{t}(f).
\]

\end{prop}
For lower-semicontinuous quasi-convex functions we have a similar
result, with the lower level sets $\KL_{t}(\phi)$ playing the role
of the upper level sets $\KU_{t}(f)$.
\begin{proof}
If $y_{0}\in\KU_{t}(f)$ and $z_{0}\in\KU_{t}(g)$ then 
\begin{eqnarray*}
(f\oplus g)(y_{0}+z_{0}) & = & \sup_{y\in\R^{n}}\min\left\{ f(y),g(y_{0}+z_{0}-y)\right\} \\
 & \ge & \min\left\{ f(y_{0}),g(y_{0}+z_{0}-y_{0})\right\} )\\
 & \ge & \min\left\{ t,t\right\} =t,
\end{eqnarray*}
 so $y_{0}+z_{0}\in\KU_{t}(f+g)$, and it follows that 
\[
\KU_{t}(f\oplus g)\supseteq\KU_{t}(f)+\KU_{t}(g).
\]
 For the other inclusion, fix $t>0$ and assume $x_{0}\in\KU_{t}(f+g)$.
This means that 
\[
\left(f\oplus g\right)(x_{0})=\sup_{y\in\R^{n}}\min\left\{ f(y),g(x_{0}-y)\right\} \ge t,
\]
 so we can choose a sequence $\left\{ y_{m}\right\} _{m=1}^{\infty}$
such that $f(y_{m})>t-\frac{1}{m}$ and $g(x_{0}-y_{m})>t-\frac{1}{m}$.
Fix $m_{0}$ large enough that $t-\frac{1}{m_{0}}>0$. For every $m\ge m_{0}$
we have $y_{m}\in\KU_{t-\frac{1}{m_{0}}}(f)$, and since $\KU_{t-\frac{1}{m_{0}}}(f)$
is compact we can assume without loss of generality that $y_{m}\to y$
as $m\to\infty$. By the upper semi-continuity of $f$ and $g$ we
get 
\begin{eqnarray*}
f(y) & \ge & \limsup_{m\to\infty}f(y_{m})\ge\limsup_{m\to\infty}t-\frac{1}{m}=t\\
g(x_{0}-y) & \ge & \limsup_{m\to\infty}g(x_{0}-y_{m})\ge\limsup_{m\to\infty}t-\frac{1}{m}=t,
\end{eqnarray*}
 and then $x_{0}=y+\left(x_{0}-y\right)\in\KU_{t}(f)+\KU_{t}(g)$
. This proves the first assertion of the proposition. The assertion
$\KU_{t}(\lambda\odot f)=\lambda\cdot\KU_{t}(f)$ is trivial.
\end{proof}

The compactness assumption in Proposition \ref{prop:sum-by-levels}
is necessary. To see this, define $f,g\in\textrm{QC}\!\left(\R\right)$
by $f(x)=\frac{\pi}{2}+\arctan x$ and $g(x)=\frac{\pi}{2}-\arctan x$.
Then 
\[
\left(f\oplus g\right)(x)=\sup_{y\in\R^{n}}\min\left\{ \frac{\pi}{2}+\arctan y,\frac{\pi}{2}+\arctan\left(y-x\right)\right\} =\pi
\]
 for all $x$, and we have 
\[
\KU_{\pi}(f)+\KU_{\pi}(g)=\emptyset+\emptyset=\emptyset\ne\R=\KU_{\pi}(f\oplus g).
\]
This is a minor detail, however, as our main interest is in functions
$f\in\qc$ such that $0<\int f<\infty$, and for such functions the
level sets $\KU_{t}(f)$ are indeed compact. It is also easy to check
and useful to notice that even without any compactness assumptions,
we still have 
\[
\left\{ x:\ \left(f\oplus g\right)(x)>t\right\} \subseteq\KU_{t}(f)+\KU_{t}(g).
\]

We will now prove that the sum of convex functions is indeed convex:
\begin{prop}
For $\phi,\psi\in\cvx$ the function $\phi\oplus\psi$ is convex.
If, in addition, the sets $\KL_{t}(\phi)$ are compact, then $\phi\oplus\psi$
is lower semicontinuous, so $\phi\oplus\psi\in\cvx$.\end{prop}
\begin{proof}
We will verify directly that $\phi\oplus\psi$ is convex. Fix $x_{0},x_{1}\in\R^{n}$
and $0<\lambda<1$. For every $\epsilon>0$ we can find $y_{0},y_{1}\in\R^{n}$
such that 
\begin{eqnarray*}
\left(\phi\oplus\psi\right)(x_{0}) & = & \inf_{y\in\R^{n}}\max\left\{ \phi(y),\psi(x_{0}-y)\right\} \ge\max\left\{ \phi(y_{0}),\psi(x_{0}-y_{0})\right\} -\epsilon\\
\left(\phi\oplus\psi\right)(x_{1}) & = & \inf_{y\in\R^{n}}\max\left\{ \phi(y),\psi(x_{1}-y)\right\} \ge\max\left\{ \phi(y_{1}),\psi(x_{1}-y_{1})\right\} -\epsilon.
\end{eqnarray*}
 Define $x_{\lambda}=(1-\lambda)x_{0}+\lambda x_{1}$ and $y_{\lambda}=(1-\lambda)y_{0}+\lambda y_{1}$.
Notice that 
\begin{eqnarray*}
\phi(y_{\lambda}) & \le & (1-\lambda)\phi(y_{0})+\lambda\phi(y_{1})\\
 & \le & (1-\lambda)\max\left\{ \phi(y_{0}),\psi(x_{0}-y_{0})\right\} +\lambda\max\left\{ \phi(y_{1}),\psi(x_{1}-y_{1})\right\} \\
 & \le & (1-\lambda)\left[\left(\phi\oplus\psi\right)(x_{0})+\epsilon\right]+\lambda\left[\left(\phi\oplus\psi\right)(x_{1})+\epsilon\right]\\
 & = & (1-\lambda)\left[\left(\phi\oplus\psi\right)(x_{0})\right]+\lambda\left[\left(\phi\oplus\psi\right)(x_{1})\right]+\epsilon,
\end{eqnarray*}
 and similarly 
\begin{eqnarray*}
\psi(x_{\lambda}-y_{\lambda}) & \le & (1-\lambda)\psi(x_{0}-y_{0})+\lambda\psi(x_{1}-y_{1})\\
 & \le & (1-\lambda)\left[\left(\phi\oplus\psi\right)(x_{0})\right]+\lambda\left[\left(\phi\oplus\psi\right)(x_{1})\right]+\epsilon.
\end{eqnarray*}
 Therefore 
\begin{eqnarray*}
\left(\phi\oplus\psi\right)(x_{\lambda}) & = & \inf_{y\in\R^{n}}\max\left\{ \phi(y),\psi(x_{\lambda}-y)\right\} \le\max\left\{ \phi(y_{\lambda}),\psi(x_{\lambda}-y_{\lambda})\right\} \\
 & \le & (1-\lambda)\left[\left(\phi\oplus\psi\right)(x_{0})\right]+\lambda\left[\left(\phi\oplus\psi\right)(x_{1})\right]+\epsilon.
\end{eqnarray*}
 Taking $\epsilon\to0$, we obtain the result. 

If the level sets $\KL_{t}(\phi)$ are compact, one may apply Proposition
\ref{prop:sum-by-levels} and conclude that for every $t\in\R$ 
\[
\KL_{t}(\phi\oplus\psi)=\KL_{t}(\phi)+\KL_{t}(\psi).
\]
 Since $\KL_{t}(\phi)$ is compact and $\KL_{t}(\psi)$ is closed,
their Minkowski sum $\KL_{t}(\phi\oplus\psi)$ is closed as well.
This implies that $\phi\oplus\psi$ is lower semicontinuous, and the
proof is complete.
\end{proof}
Again, the compactness assumption is necessary. This is not a surprise,
because it is well known that even the corresponding theorem for convex
bodies fails without compactness: Define 
\begin{eqnarray*}
K_{1} & = & \left\{ (x,y)\in\R^{2}:\ x>0,\ y\ge\frac{1}{x}\right\} \\
K_{2} & = & \left\{ (x,y)\in\R^{2}:\ x>0,\ y\le-\frac{1}{x}\right\} .
\end{eqnarray*}
 Then $K_{1}$ and $K_{2}$ are closed, convex sets, but their sum
$K_{1}+K_{2}=\left\{ (x,y)\in\R^{2}:\ x>0\right\} $ is not closed.
If we now define 
\[
\phi_{i}(x,y)=\o_{K_{i}}^{\infty}(x,y)=\begin{cases}
0 & (x,y)\in K_{i}\\
\infty & \text{otherwise, }
\end{cases}
\]
 then each $\phi_{i}$ is lower semicontinuous, but $\phi_{1}\oplus\phi_{2}=\o_{K_{1}+K_{2}}^{\infty}$
is not.

We are now ready to prove Minkowski's theorem for our addition:
\begin{thm}
\label{thm:polynomial}Fix $f_{1},f_{2},\ldots,f_{m}\in\qc$. Then
the function $F:\left(\R^{+}\right)^{m}\to[0,\infty]$, defined by
\[
F(\epsilon_{1},\epsilon_{2},\ldots,\epsilon_{m})=\int\left[\left(\epsilon_{1}\odot f_{1}\right)\oplus\left(\epsilon_{2}\odot f_{2}\right)\oplus\cdots\oplus\left(\epsilon_{m}\odot f_{m}\right)\right]
\]
 is a homogenous polynomial of degree $n$, with non-negative coefficients.
If we write 
\[
F(\epsilon_{1},\epsilon_{2},\ldots,\epsilon_{m})=\sum_{i_{1},i_{2},\ldots,i_{n}=1}^{m}\epsilon_{i_{1}}\epsilon_{i_{2}}\cdots\epsilon_{i_{n}}\cdot V(f_{i_{1}},f_{i_{2}}\ldots,f_{i_{n}})
\]
 for a symmetric function $V$, then 
\[
V(f_{1},f_{2},\ldots,f_{n})=\int_{0}^{\infty}V\left(\KU_{t}\left(f_{1}\right),\KU_{t}\left(f_{2}\right),\ldots,\KU_{t}\left(f_{n}\right)\right)dt.
\]
\end{thm}
\begin{proof}
Define $h=\left(\epsilon_{1}\odot f_{1}\right)\oplus\left(\epsilon_{2}\odot f_{2}\right)\oplus\cdots\oplus\left(\epsilon_{m}\odot f_{m}\right)$.
Using Fubini's theorem we can integrate by level sets and obtain 
\[
\int h=\int_{0}^{\infty}\left|\KU_{t}(h)\right|dt,
\]
 where $\left|\cdot\right|$ denotes the Lebesgue volume.

Applying Proposition \ref{prop:sum-by-levels}, we see that if the
upper level sets $\KU_{t}(f_{i})$ are all compact then 
\[
\KU_{t}(h)=\epsilon_{1}\KU_{t}(f_{1})+\epsilon_{2}\KU_{t}(f_{2})+\cdots+\epsilon_{m}\KU_{t}(f_{m}),
\]
 so we can integrate and obtain 
\[
\int_{0}^{\infty}\left|\KU_{t}(h)\right|dt=\int_{0}^{\infty}\left|\epsilon_{1}\KL_{t}(f_{1})+\epsilon_{2}\KL_{t}(f_{2})+\cdots+\epsilon_{m}\KL_{t}(f_{m})\right|dt.
\]

In fact, by being a bit careful, it is possible to obtain the above
formula even without assuming compactness. Indeed, by the discussion
after Proposition \ref{prop:sum-by-levels}, we see that we always
have 
\[
\KU_{t}(h)\supseteq\epsilon_{1}\KU_{t}(f_{1})+\epsilon_{2}\KU_{t}(f_{2})+\cdots+\epsilon_{m}\KU_{t}(f_{m})\supseteq\left\{ x:\ h(x)>t\right\} .
\]
 Since $\left|\KU_{t}(h)\right|=\left|\left\{ x:\ h(x)>t\right\} \right|$
for all but countably many values of $t$, we get that 
\[
\left|\KU_{t}(h)\right|=\left|\epsilon_{1}\KU_{t}(f_{1})+\epsilon_{2}\KU_{t}(f_{2})+\cdots+\epsilon_{m}\KU_{t}(f_{m})\right|
\]
 for all but countably many values of $t$, so we can still integrate
and get the formula we want.

Now we apply the classic Minkowski theorem and obtain 
\begin{eqnarray*}
\int h & = & \int_{0}^{\infty}\left|\epsilon_{1}\KL_{t}(f_{1})+\epsilon_{2}\KL_{t}(f_{2})+\cdots+\epsilon_{m}\KL_{t}(f_{m})\right|dt\\
 & = & \int_{0}^{\infty}\sum_{i_{1},i_{2},\ldots,i_{n}=1}^{m}\epsilon_{i_{1}}\epsilon_{i_{2}}\cdots\epsilon_{i_{n}}\cdot V\left(\KU_{t}\left(f_{i_{1}}\right),\KU_{t}\left(f_{i_{2}}\right),\ldots,\KU_{t}\left(f_{i_{m}}\right)\right)dt\\
 & = & \sum_{i_{1},i_{2},\ldots,i_{n}=1}^{m}\epsilon_{i_{1}}\epsilon_{i_{2}}\cdots\epsilon_{i_{n}}\cdot\int_{0}^{\infty}V\left(\KU_{t}\left(f_{i_{1}}\right),\KU_{t}\left(f_{i_{2}}\right),\ldots,\KU_{t}\left(f_{i_{m}}\right)\right)dt
\end{eqnarray*}
 which is exactly what we wanted.
\end{proof}
From Theorem \ref{thm:polynomial} the following definition becomes
natural:
\begin{defn}
\label{def:mixed-integrals}If $f_{1},\ldots,f_{n}\in\qc$ we define
their mixed integral as 
\[
V(f_{1},f_{2},\ldots,f_{n})=\int_{0}^{\infty}V\left(\KU_{t}\left(f_{1}\right),\KU_{t}\left(f_{2}\right),\ldots,\KU_{t}\left(f_{n}\right)\right)dt.
\]

\end{defn}
The mixed integrals $V(f_{1},f_{2},\ldots,f_{n})$ are exactly the
polarization we were looking for -- it is a symmetric, multilinear
functional such that $V(f,f,\ldots,f)=\int f$ for all $f\in\lc$. 

We will now give a couple of examples of mixed integrals:
\begin{example}
One can view mixed volumes as a special case of mixed integrals. Fix
bodies $K_{1},\ldots,K_{n}$ and a parameter $p>0$, and define $f_{i}(x)=\exp\left(-\left\Vert x\right\Vert _{K_{i}}^{p}\right)\in\qc$.
Then for every $0<t<1$ we have 
\[
\KU_{t}(f_{i})=\left\{ x:\ \exp\left(-\left\Vert x\right\Vert _{K_{i}}^{p}\right)\ge t\right\} =\left\{ x:\ \left\Vert x\right\Vert _{K_{i}}\le\left(\log\frac{1}{t}\right)^{\frac{1}{p}}\right\} =\left(\log\frac{1}{t}\right)^{\frac{1}{p}}\cdot K_{i}.
\]
 Hence, 
\begin{eqnarray*}
V(f_{1},\ldots,f_{n}) & = & \int_{0}^{1}V\left(\left(\log\frac{1}{t}\right)^{\frac{1}{p}}K_{1},\ldots,\left(\log\frac{1}{t}\right)^{\frac{1}{p}}K_{n}\right)dt\\
 & = & \left[\int_{0}^{1}\left(\log\frac{1}{t}\right)^{\frac{n}{p}}dt\right]\cdot V(K_{1},\ldots,K_{n}),
\end{eqnarray*}
 so the mixed integral of $f_{1},\ldots,f_{n}$ is the same as the
mixed volume of $K_{1},\ldots,K_{n}$, up to normalization. In particular,
taking $p\to\infty$ we get 
\[
V\left(\o_{K_{1}},\o_{K_{2}},\ldots,\o_{K_{n}}\right)=V\left(K_{1},\ldots,K_{n}\right),
\]
 which can also be seen directly from the definition.

Of course, one can make this example even more general, by choosing
$f_{i}$ to be any integrable function such that $\KU_{t}(f_{i})$
is always homothetic to $K_{i}$.
\end{example}

\begin{example}
\label{example:quermass}Fix any $f\in\qc$, and assume for simplicity
that $f(x)\le1$ for all $x$. Also, fix a convex set $K\in\kn$.
It is not hard to see that for every $\epsilon>0$ we have 
\[
f\oplus\left(\epsilon\odot\o_{K}\right)=f\star\left(\epsilon\cdot\o_{K}\right)=\sup_{y\in\epsilon K}f(x-y).
\]
 Here $\star$ is the sup-convolution operation described in the introduction,
and $\cdot$ is the induced homothety operation defined by $\left(\lambda\cdot g\right)(x)=g\left(\frac{x}{\lambda}\right)^{\lambda}$.
It follows from Theorem \ref{thm:polynomial} that the integral $\int f\star\left(\epsilon\cdot\o_{K}\right)$
is a polynomial in $\epsilon$. In other words, in this restricted
case, we obtain a Minkowski type theorem for the sup-convolution operation. 

In particular, define the $\epsilon$-extension of $f$ to be 
\[
f_{\epsilon}(x)=\sup_{\left|y\right|\le\epsilon}f(x+y).
\]
Since $f_{\epsilon}=f\oplus\left(\epsilon\odot\o_{D}\right)$, where
$D$ is the Euclidean ball, the integral $\int f_{\epsilon}$ is a
polynomial in $\epsilon$. Its coefficients, 
\[
W_{k}(f)=V(\underbrace{f,f,\ldots f}_{n-k\text{ times}},\underbrace{\o_{D},\o_{D},\ldots,\o_{D}}_{k\text{ times}})
\]
will be called the quermassintegrals of $f$. This is in analogy to
the classic quermassintegrals, which are defined for every $K\in\kn$
by 
\[
W_{k}(K)=V(\underbrace{K,K,\ldots K}_{n-k\text{ times}},\underbrace{D,D,\ldots,D}_{k\text{ times}}).
\]
This example was discovered independently by Colesanti \cite{colesanti_what_2012},
who also proved that these numbers share several important properties
of the classic quermassintegrals. 

Of course, there is no reason to use only one convex body. One easily
checks that for every $f\in\qc$ and every convex sets $K_{1},K_{2},\ldots,K_{m}\in\kn$,
the integral 
\[
\int\left[f\star\left(\epsilon_{1}\odot\o_{K_{1}}\right)\star\left(\epsilon_{2}\odot\o_{K_{2}}\right)\star\cdots\star\left(\epsilon_{m}\odot\o_{K_{m}}\right)\right]
\]
 is a polynomial in $\epsilon_{1},\epsilon_{2},\ldots,\epsilon_{m}$. 
\end{example}
We see that the operation $\oplus$ satisfies a Minkowski theorem,
while the standard inf-convolution $\square$ does not. It is interesting
to notice that, nonetheless, the operations $\oplus$ and $\square$
are not so different. In fact, for every positive convex functions
$\phi$ and $\psi$ we have 
\[
\frac{1}{2}\left(\phi\square\psi\right)(x)\le\left(\phi\oplus\psi\right)(x)\le\left(\phi\square\psi\right)(x).
\]
 This follows from the following proposition:
\selectlanguage{american}%
\begin{prop}
\label{prop:old-new-relation}Fix convex functions $\phi_{1},\phi_{2},\ldots\phi_{k}\in\cvx$
such that $\phi_{i}(x)\ge0$ for all $1\le i\le k$ and $x\in\R^{n}$.
For $\lambda_{1},\lambda_{2},\ldots,\lambda_{k}>0$, define
\begin{eqnarray*}
g_{1}(x) & = & \left(\lambda_{1}\cdot\phi_{1}\square\lambda_{2}\cdot\phi_{2}\square\cdots\square\lambda_{k}\cdot\phi_{k}\right)(x)\\
g_{2}(x) & = & \left(\lambda_{1}\odot\phi_{1}\oplus\lambda_{2}\odot\phi_{2}\oplus\cdots\oplus\lambda_{k}\odot\phi_{k}\right)(x)
\end{eqnarray*}
Then 
\[
\left(\sum_{i}\lambda_{i}\right)^{-1}g_{1}(x)\le g_{2}(x)\le\left(\min\lambda_{i}\right)^{-1}g_{1}(x).
\]
\end{prop}
\selectlanguage{english}%
\begin{proof}
By definition we have 
\begin{eqnarray*}
g_{1}(x) & = & \inf_{\sum\lambda_{i}y_{i}=x}\sum_{i=1}^{k}\lambda_{i}f_{i}(y_{i}),\\
g_{2}(x) & = & \inf_{\sum\lambda_{i}y_{i}=x}\max_{1\le i\le k}\left\{ f_{i}(y_{i})\right\} .
\end{eqnarray*}
\foreignlanguage{american}{For every non-negative numbers $a_{1},\ldots,a_{k}$
and positive coefficients $\lambda_{1},\ldots,\lambda_{k}$ we have
\[
\frac{\sum_{i}\lambda_{i}a_{i}}{\sum_{i}\lambda_{i}}\le\max_{i}a_{i}\le\frac{\sum_{i}\lambda_{i}a_{i}}{\min\lambda_{i}},
\]
 and the result follows immediately.}
\end{proof}
To conclude this section, let us digress slightly and discuss a relation
between the above constructions and the notion of polarity. From Proposition
\ref{prop:sum-by-levels} we see that our sum $\oplus$ operates on
level sets: In order to find $\underline{K}_{t}(\phi\oplus\psi)$
it is enough to know $\underline{K}_{t}(\phi)$ and $\underline{K}_{t}(\psi)$.
Another important structure in convexity is that of polarity. Let
$\phi:\R^{n}\to[0,\infty]$ be a lower semicontinuous convex function
with $\phi(0)=0$ -- such functions are called geometric convex functions.
The polar function of $\phi$, which we will denote by $\phi^{\circ}$,
is defined using the so-called $\mathcal{A}$-transform of $\phi$,
\[
\phi^{\circ}(x):=\left(\mathcal{A}\phi\right)(x)=\sup_{y\in\R^{n}}\frac{\left\langle x,y\right\rangle -1}{\phi(y)}
\]
 (see \cite{artstein-avidan_hidden_2011} for a detailed discussion.
See also \cite{rockafellar_convex_1970} section 15, as well as \cite{milman_geometrization_2008}
for historical remarks). Polarity does not work on level sets. However,
it is interesting to notice that it is ``almost'' the case:
\begin{prop}
\label{prop:duality}For every geometric convex function $\phi$ and
every $t>0$ we have 
\[
\KL_{1/t}(\phi)^{\circ}\subseteq\KL_{t}(\phi^{\circ})\subseteq2\KL_{1/t}(\phi)^{\circ}.
\]

\end{prop}
In an informal way, one can say that the polars of the level sets
are ``almost'' the level sets of the polar. 
\begin{proof}
The easier inclusion is the right one. Assume $x\in\KL_{t}(\phi^{\circ})$
and take any $y\in\KL_{1/t}(\phi)$. We have 
\[
t\ge\phi^{\circ}(x)\ge\frac{\left\langle x,y\right\rangle -1}{\phi(y)}\ge\frac{\left\langle x,y\right\rangle -1}{1/t}=t\left(\left\langle x,y\right\rangle -1\right),
\]
and when we divide by $t$ we get $\left\langle x,y\right\rangle -1\le1$,
which implies $\left\langle \frac{x}{2},y\right\rangle \le1$. Since
$y\in\KL_{1/t}(\phi)$ was arbitrary it follows that $\frac{x}{2}\in\KL_{1/t}(\phi)^{\circ}$,
so $x\in2\KL_{1/t}(\phi)^{\circ}.$

For the left inclusion, fix $x\in\KL_{1/t}(\phi)^{\circ}$ and take
any $y\in\R^{n}$. If $\phi(y)\le\frac{1}{t}$ then $y\in\KL_{1/t}(\phi)$,
so $\left\langle x,y\right\rangle \le1$ and 
\[
\frac{\left\langle x,y\right\rangle -1}{\phi(y)}\le0\le t.
\]
 If, on the other hand, $\phi(y)=s>\frac{1}{t}$, then $st>1$, and
using the convexity of $\phi$ we get 
\begin{eqnarray*}
\phi\left(\frac{y}{st}\right) & = & \phi\left(\frac{1}{st}\cdot y+\left(1-\frac{1}{st}\right)\cdot0\right)\le\frac{1}{st}\phi(y)+\left(1-\frac{1}{st}\right)\phi(0)=\frac{1}{t}.
\end{eqnarray*}
 Hence $\frac{y}{st}\in L_{1/t}\left(\phi\right)$, so $\left\langle x,\frac{y}{st}\right\rangle \le1$,
and then 
\[
\frac{\left\langle x,y\right\rangle -1}{\phi(y)}\le\frac{st-1}{s}=t-\frac{1}{s}\le t.
\]
 All together we get 
\[
\phi^{\circ}(x)=\sup_{y\in\R^{n}}\frac{\left\langle x,y\right\rangle -1}{\phi(y)}\le t,
\]
 so $x\in\KL_{t}(\phi^{\circ})$, like we wanted.
\end{proof}
Assume $\phi$ is any function which is geometric, quasi-convex and
lower semicontinuous. We define its dual function $\phi^{\ast}$ via
the relation 
\[
\KL_{t}(\phi^{\ast})=\KL_{1/t}(\phi)^{\circ}
\]
 for any $t>0$. It is easy to see that $\phi^{\ast}$ is also geometric,
quasi-convex and lower semicontinuous. Furthermore, we have $\left(\phi^{\ast}\right)^{\ast}=\phi$,
so the operation $\ast$ really defines a duality relation on quasi-convex
functions. Proposition \ref{prop:duality} tells us that for convex
functions, the operations $\circ$ and $\ast$ are very similar.

\section{\label{sec:Inequalities}Rearrangement inequalities}

In this section we will generalize several classic inequalities concerning
mixed volumes to the realm of quasi-concave functions. For simplicity,
we will always assume our quasi-concave functions are geometric:
\begin{defn}
A function $f\in\qc$ is called geometric if 
\[
\max_{x\in\R^{n}}f(x)=f(0)=1.
\]
 The class of all geometric quasi-concave functions will be denoted
by $\qco$. We define the class $\lco$ of geometric log-concave functions
in a similar way.

In section \ref{sec:LC-inequalities}, the fact that the functions
involved are geometric will play a crucial role. Here, however, this
is merely a matter of convenience, allowing us to ignore some technical
details -- many of the results will remain true even without this
assumption.
\end{defn}
Remember from the introduction that our first goal is to state and
prove an extension of the isoperimetric inequality to our case: we
want to give a lower bound on 
\[
S(f)=n\cdot W_{1}(f)
\]
in terms of the integral $\int f$ (the notation $W_{k}$ appeared
in Example \ref{example:quermass}). Unfortunately, this is impossible:
In Remark \ref{rem:counter-isoperimetric} we will construct a sequence
$f_{k}\in\text{QC}_{0}\left(\R^{2}\right)$ with $\int f_{k}=1$,
but $S(f_{k})\stackrel{k\to\infty}{\longrightarrow}0$. Hence we follow
another route and define:
\begin{defn}
\label{def:SDR}
\begin{enumerate}
\item For a compact $K\in\kn$, define 
\[
K^{\ast}=\left(\frac{\vol(K)}{\vol(D)}\right)^{\frac{1}{n}}D.
\]
In other words, $K^{\ast}$ is the Euclidean ball with the same volume
as $K$. 
\item For $f\in\qco$ with compact upper level sets, define its symmetric
decreasing rearrangement $f^{\ast}$ using the relation 
\[
\KU_{t}\left(f^{\ast}\right)=\KU_{t}(f)^{\ast}.
\]
 
\end{enumerate}
\end{defn}
It is easy to see that this definition really defines a unique function
$f^{\ast}\in\qco$, which is rotation invariant.

Since $K^{\ast}$ is a ball with the same volume as $K$, the isoperimetric
inequality tells us that $S(K)\ge S(K^{\ast})$ for every convex body
$K\in\kn$. This means we can think about the isoperimetric inequality
as a rearrangement inequality, and this point of view can be extended
to quasi-concave functions:
\begin{prop}
\label{prop:qc-isoperimetric}If $f\in\qco$ has compact level sets,
then $S(f)\ge S(f^{\ast})$, with equality if and only if $f$ is
rotation invariant.\end{prop}
\begin{proof}
Notice that for every function $g\in\qco$ we have 
\begin{eqnarray*}
S(g) & = & n\cdot W_{1}(g)=n\cdot V(\underbrace{g,g,\ldots,g}_{n-1\text{ times}},\o_{D})\\
 & = & \int_{0}^{1}n\cdot V\left(\KU_{t}(g),\KU_{t}(g),\ldots,\KU_{t}(g),D\right)dt=\int_{0}^{1}S(\KU_{t}(g))dt.
\end{eqnarray*}
Using the classic isoperimetric inequality we get
\[
S(f^{\ast})=\int_{0}^{1}S(\KU_{t}(f^{\ast}))dt=\int_{0}^{1}S(\KU_{t}(f)^{\ast})dt\le\int_{0}^{1}S(\KU_{t}(f))dt=S(f),
\]
 which is what we wanted. 

If $S(f^{\ast})=S(f)$ then $S(\KU_{t}(f)^{\ast})=S(\KU_{t}(f))$
for all $t$. Again by the classic isoperimetric inequality this implies
that $\KU_{t}(f)$ is always a ball, so $f$ is rotation invariant. 
\end{proof}
In classic convexity, the isoperimetric inequality follows easily
from the Brunn-Minkowski theorem, which states that for any (say convex)
bodies $A,B\in\kn$, we have
\[
\vol(A+B)^{\frac{1}{n}}\ge\vol(A)^{\frac{1}{n}}+\vol(B)^{\frac{1}{n}}.
\]
 The Brunn-Minkowski theorem can also be written as a rearrangement
inequality: $(A+B)^{\ast}\supseteq A^{\ast}+B^{\ast}$. The corresponding
result for quasi-concave functions is
\begin{prop}
\label{prop:BM}If $f,g\in\qco$ have compact level sets, then $\left(f\oplus g\right)^{\ast}\ge f^{\ast}\oplus g^{\ast}.$ 
\end{prop}
In particular, we have 
\[
\int f\oplus g=\int\left(f\oplus g\right)^{\ast}\ge\int f^{\ast}\oplus g^{\ast}.
\]

\begin{proof}
It is enough to prove that $\KU_{t}\left(\left(f\oplus g\right)^{\ast}\right)\supseteq\KU_{t}\left(f^{\ast}\oplus g^{\ast}\right)$
for all $t$ . But 
\begin{eqnarray*}
\KU_{t}\left(\left(f\oplus g\right)^{\ast}\right) & = & \KU_{t}\left(f\oplus g\right)^{\ast}=\left(\KU_{t}(f)+\KU_{t}(g)\right)^{\ast}\\
 & \supseteq & \KU_{t}(f)^{\ast}+\KU_{t}(g)^{\ast}=\KU_{t}(f^{\ast})+\KU_{t}(g^{\ast})\\
 & = & \KU_{t}\left(f^{\ast}\oplus g^{\ast}\right),
\end{eqnarray*}
 so the result holds.
\end{proof}
We would now like to take even more general inequalities and cast
them to our setting. For example, the most general Brunn-Minkowski
inequality for mixed volumes states that for every $A,B,K_{1},K_{2},\ldots,K_{n-m}\in\kn$
we have 
\begin{eqnarray*}
V(\underbrace{A+B,\ldots,A+B}_{m\text{ times}},K_{1},\ldots,K_{n-m})^{\frac{1}{m}} & \ge & V(A,\ldots,A,K_{1},\ldots,K_{n-m})^{\frac{1}{m}}+\\
 &  & \ +V(B,\ldots,B,K_{1},\ldots,K_{n-m})^{\frac{1}{m}}.
\end{eqnarray*}
 In order to write such an inequality in our language, we will need
to define a generalized concept of rearrangement:
\begin{defn}

\begin{enumerate}
\item A size functional is a function $\Phi:\kn\to[0,\infty]$ of the form
\[
\Phi(A)=V(\underbrace{A,\ldots,A}_{m\text{ times}},K_{1},\ldots,K_{n-m}),
\]
 for fixed compact bodies $K_{1},K_{2},\ldots,K_{n-m}$ with non-empty
interior. We will say that $\Phi$ is of degree $m$.
\item If $K\in\kn$ is compact and $\Phi$ is a size functional of degree
$m$, define 
\[
K^{\Phi}=\left(\frac{\Phi(K)}{\Phi(D)}\right)^{\frac{1}{m}}\cdot D.
\]
 If, in addition, $f\in\qco$ has compact level sets, define $f^{\Phi}\in\qco$
by the relation
\[
\KU_{t}\left(f^{\Phi}\right)=\KU_{t}(f)^{\Phi}.
\]

\end{enumerate}
\end{defn}
In particular, we have $K^{\vol}=K^{\ast}$ and $f^{\vol}=f^{\ast}$
for a convex body $K$ and a quasi-concave function $f$. Notice that
$K^{\Phi}$ is the Euclidean ball of the same ``size'' as $K$,
where size is defined using the functional $\Phi$.

For functions the intuition is similar. If $\Phi:\kn\to[0,\infty]$
is a size functional, we can extend the domain of $\Phi$ to all of
$\qco$ in a natural way: If
\[
\Phi(A)=V(\underbrace{A,\ldots,A}_{m\text{ times}},K_{1},\ldots,K_{n-m}),
\]
 then 
\[
\Phi(f)=V(\underbrace{f,\ldots,f}_{m\text{ times}},\o_{K_{1}},\ldots,\o_{K_{n-m}}).
\]
 Notice that for every quasi-concave function $g\in\qco$ we have
\[
\Phi(g)=\int_{0}^{1}\Phi\left(\KU_{t}(g)\right)dt,
\]
 so in particular 
\begin{eqnarray*}
\Phi(f^{\Phi}) & = & \int_{0}^{1}\Phi\left(\KU_{t}(f^{\Phi})\right)dt=\int_{0}^{1}\Phi\left(\KU_{t}(f)^{\Phi}\right)dt\\
 & = & \int_{0}^{1}\Phi\left(\KU_{t}(f)\right)dt=\Phi(f).
\end{eqnarray*}
 This means that $f^{\Phi}$ is a rotation invariant function with
the same ``size'' as $f$. 

Now we can write the general Brunn-Minkowski inequality as a generalized
rearrangement inequality, both for convex bodies and for convex functions:
\begin{thm}
\label{prop:gen-BM}Let $\Phi$ be a size functional. Then
\begin{enumerate}
\item $\left(A+B\right)^{\Phi}\supseteq A^{\Phi}+B^{\Phi}$ for every compact
convex bodies $A,B\in\kn$.
\item $\left(f\oplus g\right)^{\Phi}\ge f^{\Phi}\oplus g^{\Phi}$ for every
quasi-concave functions $f,g\in\qco$ with compact level sets.
\end{enumerate}
\end{thm}
\begin{proof}
First we deal with the case of bodies, where the proposition is just
a restatement of the generalized Brunn-Minkowski inequality: Notice
that $\left(A+B\right)^{\Phi}$ is a ball of radius 
\[
\left(\frac{\Phi(A+B)}{\Phi(D)}\right)^{\frac{1}{m}},
\]
 where $m$ is the degree of $\Phi$. Similarly, $A^{\Phi}+B^{\Phi}$
is a ball of radius 
\[
\left(\frac{\Phi(A)}{\Phi(D)}\right)^{\frac{1}{m}}+\left(\frac{\Phi(B)}{\Phi(D)}\right)^{\frac{1}{m}},
\]
 so the result follows immediately. 

For quasi-concave functions, we argue just like in Proposition \ref{prop:BM}:
\begin{eqnarray*}
\KU_{t}\left(\left(f\oplus g\right)^{\Phi}\right) & = & \KU_{t}\left(f\oplus g\right)^{\Phi}=\left(\KU_{t}(f)+\KU_{t}(g)\right)^{\Phi}\supseteq\KU_{t}(f)^{\Phi}+\KU_{t}(g)^{\Phi}\\
 & = & \KU_{t}(f^{\Phi})+\KU_{t}(g^{\Phi})=\KU_{t}\left(f^{\Phi}\oplus g^{\Phi}\right),
\end{eqnarray*}
 so $\left(f\oplus g\right)^{\Phi}\ge f^{\Phi}+g^{\Phi}$ like we
wanted.
\end{proof}
Again, we see as a corollary that 
\[
\Phi(f\oplus g)=\Phi\left(\left(f\oplus g\right)^{\Phi}\right)\ge\Phi\left(f^{\Phi}\oplus g^{\Phi}\right).
\]

Other geometric inequalities can be written in the same form as well.
The Alexandrov inequalities for quermassintegrals state that for every
$K\in\kn$ and every $0\le i<j<n$ we have 
\[
\left(\frac{W_{j}(K)}{W_{j}(D)}\right)^{\frac{1}{n-j}}\ge\left(\frac{W_{i}(K)}{W_{i}(D)}\right)^{\frac{1}{n-i}},
\]
 with equality if and only if $K$ is a ball. In the language of rearrangements,
we can write:
\begin{prop}
\label{prop:Alexandrov}Fix $0\le i<j<n$. Then 
\begin{enumerate}
\item $K^{W_{j}}\supseteq K^{W_{i}}$ for every compact, convex body $K\in\kn$.
If $K^{W_{j}}=K^{W_{i}}$ then $K$ is a ball.
\item $f^{W_{j}}\ge f^{W_{i}}$ for every quasi-concave $f\in\qco$ with
compact level sets. If $f^{W_{j}}=f^{W_{i}}$ , then $f$ is rotation
invariant.
\end{enumerate}
\end{prop}
\begin{proof}
Since $K^{W_{j}}$ is simply a ball of radius 
\[
\left(\frac{W_{j}(K)}{W_{j}(D)}\right)^{\frac{1}{n-j}},
\]
 the first claim is just a reformulation of the Alexandrov inequalities.
The second claim will follow easily by comparing level sets:

\[
\KU_{t}(f^{W_{j}})=\KU_{t}(f)^{W_{j}}\supseteq\KU_{t}(f)^{W_{i}}=\KU(f^{W_{i}}).
\]

\end{proof}
Again, we see as a corollary that if $f\in\qco$ then for every $j>i$
the function $g=f^{W_{i}}$ is rotation invariant, and satisfies 
\begin{eqnarray*}
W_{i}(g) & = & W_{i}(f^{W_{i}})=W_{i}(f)\\
W_{j}(g) & = & W_{j}(f^{W_{i}})\le W_{j}(f^{W_{j}})=W_{j}(f).
\end{eqnarray*}
 The case $i=0$, $j=1$ is just the isoperimetric inequality proven
earlier.

Now we would like to prove a version of the powerful Alexandrov-Fenchel
inequality. We will state and prove the proposition, and the proof
will explain in what way this is really an Alexandrov-Fenchel type
theorem
\begin{thm}
\label{prop:AF}Fix a size functional 
\[
\Phi(A)=V(\underbrace{A,\ldots,A}_{m\text{ times}},K_{1},\ldots,K_{n-m}).
\]
 Then:
\begin{enumerate}
\item For every compact bodies $A_{1},\ldots,A_{m}\in\kn$ we have 
\[
V(A_{1},\ldots,A_{m},K_{1},\ldots,K_{n-m})\ge V(A_{1}^{\Phi},\ldots,A_{m}^{\Phi},K_{1},\ldots,K_{n-m}).
\]

\item For every functions $f_{1},\ldots,f_{m}\in\qco$ with compact level
sets we have 
\[
V(f_{1},\ldots,f_{m},\o_{K_{1}},\ldots,\o_{K_{n-m}})\ge V(f_{1}^{\Phi},\ldots,f_{m}^{\Phi},\o_{K_{1}},\ldots,\o_{K_{n-m}}).
\]

\end{enumerate}
\end{thm}
\begin{proof}
For (i), notice that the right hand side is actually 
\[
V\left(\left(\frac{\Phi(A_{1})}{\Phi(D)}\right)^{\frac{1}{m}}D,\ldots,\left(\frac{\Phi(A_{m})}{\Phi(D)}\right)^{\frac{1}{m}}D,K_{1},\ldots,K_{n-m}\right),
\]
which is equal to 
\[
\prod_{i=1}^{m}\left(\frac{\Phi(A_{i})}{\Phi(D)}\right)^{\frac{1}{m}}\cdot V(D,D,\ldots,D,K_{1},\ldots,K_{n-m})=\prod_{i=1}^{m}\Phi(A_{i})^{\frac{1}{m}}.
\]
 Therefore the inequality we need to prove is 
\[
V(A_{1},\ldots A_{m},K_{1},\ldots,K_{n-m})\ge\prod_{i=1}^{m}\Phi(A_{i})^{\frac{1}{m}},
\]
 which is exactly the Alexandrov-Fenchel inequality. Hence the result
holds (sometimes the Alexandrov-Fenchel inequality is stated only
for $m=2$, but the general case is also well known and follows by
induction). 

For (ii), we calculate and get 
\begin{eqnarray*}
V(f_{1}^{\Phi},\ldots,f_{m}^{\Phi},\o_{K_{1}},\ldots,\o_{K_{n-m}}) & = & \int_{0}^{1}V(\KU_{t}(f_{1}^{\Phi}),\ldots,\KU_{t}(f_{m}^{\Phi}),K_{1},\ldots,K_{n-m})dt\\
 & = & \int_{0}^{1}V(\KU_{t}(f_{1})^{\Phi},\ldots,\KU_{t}(f_{m})^{\Phi},K_{1},\ldots,K_{n-m})dt\\
 & \le & \int_{0}^{1}V(\KU_{t}(f_{1}),\ldots,\KU_{t}(f_{m}),K_{1},\ldots,K_{n-m})dt\\
 & = & V(f_{1},\ldots,f_{m},\o_{K_{1}},\ldots,\o_{K_{n-m}}).
\end{eqnarray*}
 This completes the proof. 
\end{proof}
The case $m=n$ and $\Phi=\vol$ in the last proposition is especially
elegant, so we will state it as a corollary:
\begin{cor}
\label{cor:from-AF}For every functions $f_{1},\ldots,f_{n}\in\qco$
with compact level sets we have 
\[
V(f_{1},f_{2},\ldots,f_{n})\ge V(f_{1}^{\ast},f_{2}^{\ast},\ldots,f_{n}^{\ast}).
\]

\end{cor}
Notice that the last corollary generalizes the isoperimetric inequality.
In fact, one may define a generalized surface area as 
\[
S^{(g)}(f)=V(\underbrace{f,f,\ldots,f}_{n-1\text{ times}},g),
\]
 where $g$ is some fixed rotation invariant quasi-concave function
(natural candidates may be the exponential function $g(x)=e^{-\left|x\right|}$and
the Gaussian $g(x)=e^{-\left|x\right|^{2}/2}$). From Corollary \ref{cor:from-AF}
it follows immediately that $S^{(g)}(f)\ge S^{(g)}(f^{\ast})$ for
every $f\in\qco$.
\begin{rem}
It is clear that one may work with even more general size functionals.
A natural candidate seems to be 
\[
\Phi(A)=V(\underbrace{\o_{A},\o_{A},\ldots,\o_{A}}_{m\text{ times}},g_{1},g_{2},\ldots,g_{n-m}),
\]
 for some fixed quasi-concave functions $g_{1},\ldots,g_{n-m}\in\qco$.
There are, however, some major difficulties. The main problem is that
if we extend such a $\Phi$ to $\qco$ in the standard way, we do
not necessarily have 
\[
\Phi(f)=\Phi\left(f^{\Phi}\right)
\]
 Thus it is difficult to think of $f^{\Phi}$ as a rearrangement of
$f$ in any real sense. 

However, assume that the functions $g_{i}$ satisfy $0<\int g_{i}<\infty$
and have homothetic level sets, i.e. $\KU_{t}(g_{i})=c_{i}(t)\cdot K_{i}$
for some function $c_{i}(t)$ and some $K_{i}\in\kn$. If we define
\[
\Psi(A)=V(A,A,\ldots,A,K_{1},K_{2},\ldots,K_{n-m}),
\]
 then for every $A\in\kn$ we get 
\begin{eqnarray*}
\Phi(A) & = & \int_{0}^{1}V(A,\ldots,A,\KU_{t}(g_{1}),\ldots,\KU_{t}(g_{n-m}))dt\\
 & = & \int_{0}^{1}V(A,\ldots,A,K_{1},\ldots,K_{n-m})\cdot c_{1}(t)c_{2}(t)\cdots c_{n-m}(t)dt\\
 & = & \left[\int_{0}^{1}c_{1}(t)c_{2}(t)\cdots c_{n-m}(t)dt\right]\cdot\Psi(A)=C\cdot\Psi(A).
\end{eqnarray*}
Since $0<\int g_{i}<\infty$ we have $0<C<\infty$, and it follows
immediately that $A^{\Phi}=A^{\Psi}$. Hence $f^{\Phi}=f^{\Psi}$
for all $f\in\qco$. 

Since
\begin{eqnarray*}
\Phi(f) & = & V(f,\ldots,f,g_{1},\ldots,g_{n-m})\\
 & = & \int_{0}^{1}V(\KU_{t}(f),\ldots,\KU_{t}(f),\KU_{t}(g_{1}),\ldots,\KU_{t}(g_{n-m}))dt\\
 & = & \int_{0}^{1}c_{1}(t)\cdots c_{n-m}(t)\cdot V(\KU_{t}(f),\ldots,\KU_{t}(f),K_{1},\ldots,K_{n-m})dt\\
 & = & \int_{0}^{1}c_{1}(t)\cdots c_{n-m}(t)\cdot\Psi\left(\KU_{t}(f)\right)dt,
\end{eqnarray*}
 and similarly 
\begin{eqnarray*}
\Phi(f^{\Phi})=\Phi(f^{\Psi}) & = & \int_{0}^{1}c_{1}(t)\cdots c_{n-m}(t)\cdot\Psi\left(\KU_{t}(f^{\Psi})\right)dt\\
 & = & \int_{0}^{1}c_{1}(t)\cdots c_{n-m}(t)\cdot\Psi\left(\KU_{t}(f)^{\Psi}\right)dt,
\end{eqnarray*}
we conclude that in this specific case we do have $\Phi(f^{\Phi})=\Phi(f).$ 

Similarly, Propositions \ref{prop:gen-BM} and \ref{prop:AF} remain
true in this extended case:
\begin{prop}
\label{prop:gen-gen-BM}Let 
\[
\Phi(A)=V(\underbrace{\o_{A},\o_{A},\ldots,\o_{A}}_{m\text{ times}},g_{1},g_{2},\ldots,g_{n-m}),
\]
 be a generalized size functional, with $g_{i}\in\qco$ having homothetic
level sets. Then for every geometric quasi-concave functions $f,g\in\qco$
with compact level sets we have 
\[
\left(f\oplus g\right)^{\Phi}\ge f^{\Phi}\oplus g^{\Phi}.
\]

\end{prop}

\begin{prop}
Let $\Phi$ be a generalized size functional like in Proposition \ref{prop:gen-gen-BM}.
Then for every functions $f_{1},\ldots,f_{m}\in\qco$ with compact
level sets we have 
\[
V(f_{1},\ldots,f_{m},g_{1},\ldots,g_{n-m})\ge V(f_{1}^{\Phi},\ldots,f_{m}^{\Phi},g_{1},\ldots,g_{n-m}).
\]

\end{prop}
The proofs are simple, as one may simply replace $\Phi$ with $\Psi$.
We leave the details to the reader.
\end{rem}

\section{\label{sec:LC-inequalities}Inequalities for log-concave functions}

We now turn our attention to the log-concave case. It turns out that
for functions which are both geometric and log-concave, one can use
some 1-dimensional estimates, and prove some of the inequalities of
the previous section in a more familiar form. 

First, we will need to know that the class of log-concave functions
is preserved under rearrangements.
\begin{prop}
\label{prop:fs-log-concave}Let $\Phi$ be a size functional. If $f$
is log-concave, so is $f^{\Phi}.$\end{prop}
\begin{proof}
One can express log-concavity in terms of level-sets. A function $f$
is log-concave if and only if 
\[
\lambda\KU_{t}(f)+(1-\lambda)\KU_{s}(f)\subseteq\KU_{t^{\lambda}s^{1-\lambda}}(f)
\]
 for every $s,t>0$ and every $0<\lambda<1$. 

Using Proposition \ref{prop:gen-BM}(i), we get
\begin{eqnarray*}
\lambda\KU_{t}(f^{\Phi})+(1-\lambda)\KU_{s}(f^{\Phi}) & = & \lambda\KU_{t}(f)^{\Phi}+(1-\lambda)\KU_{s}(f)^{\Phi}\\
 & \subseteq & \left[\lambda\KU_{t}(f)+(1-\lambda)\KU_{s}(f)\right]^{\Phi}\\
 & \subseteq & \KU_{t^{\lambda}s^{1-\lambda}}(f)^{\Phi}=\KU_{t^{\lambda}s^{1-\lambda}}(f^{\Phi}),
\end{eqnarray*}
 so $f^{\Phi}$ is indeed log-concave. 
\end{proof}
Next, we will need a 1-dimensional moment estimate for log-concave
functions:
\begin{prop}
\label{prop:moments}Let $f:[0,\infty)\to[0,1]$ be a log-concave
function with $f(0)=1$. Then for every $0<k<m$ we have 
\[
\left[\frac{1}{\Gamma(m+1)}\int_{0}^{\infty}x^{m}f(x)dx\right]^{\frac{1}{m+1}}\le\left[\frac{1}{\Gamma(k+1)}\int_{0}^{\infty}x^{k}f(x)dx\right]^{\frac{1}{k+1}},
\]
 with equality if and only if $f(x)=e^{-cx}$ for some $c>0$.\end{prop}
\begin{proof}
A known result (\cite{borell_complements_1973}, see also \cite{bobkov_concentration_2011})
states that if $f:[0,\infty)\to[0,\infty)$ is log-concave, then the
function 
\[
\phi(p)=\frac{1}{\Gamma(p+1)}\int_{0}^{\infty}x^{p}f(x)dx
\]
 is log-concave on $\left(-1,\infty\right)$. Since $\phi(p)\to f(0)=1$
as $p\to-1$, we get 
\[
\phi(k)=\phi\left(\frac{k+1}{m+1}\cdot m+\frac{m-k}{m+1}\cdot(-1)\right)\ge\phi(m)^{\frac{k+1}{m+1}}\cdot1^{\frac{m-k}{m+1}}.
\]
 Hence 
\[
\phi(k)^{\frac{1}{k+1}}\ge\phi(m)^{\frac{1}{m+1}},
\]
 which is what we wanted. 
\end{proof}

We are ready to prove the Alexandrov inequalities for geometric, log-concave
functions:
\begin{thm}
\label{thm:lc-alexandrov}Define $g(x)=e^{-\left|x\right|}$. For
every $f\in\lco$ and every integers $0\le k<m<n$, we have 
\[
\left(\frac{W_{k}(f)}{W_{k}(g)}\right)^{\frac{1}{n-k}}\le\left(\frac{W_{m}(f)}{W_{m}(g)}\right)^{\frac{1}{n-m}},
\]
 with equality if and only if $f(x)=e^{-c\left|x\right|}$ for some
$c>0$. \end{thm}
\begin{proof}
By Proposition \ref{prop:Alexandrov}, $W_{m}(f^{W_{k}})\le W_{m}(f^{W_{m}})=W_{m}(f)$,
while $W_{k}(f^{W_{k}})=W_{k}(f)$. By Proposition \ref{prop:fs-log-concave},
$f^{W_{k}}\in\lco$ as well. Therefore we may replace $f$ by $f^{W_{k}}$
and assume without loss of generality that $f$ is rotation invariant. 

Write $f(x)=h\left(\left|x\right|\right),$ where $h:[0,\infty)\to[0,1]$
is a geometric, log-concave function. Let us express $W_{k}(f)$ and
$W_{m}(f)$ in terms of $h$. For every $\epsilon>0$ we have 
\[
f_{\epsilon}(x)=\begin{cases}
1 & \left|x\right|\le\epsilon\\
f\left(x-\epsilon\frac{x}{\left|x\right|}\right) & \left|x\right|>\epsilon
\end{cases}=\begin{cases}
1 & \left|x\right|\le\epsilon\\
h\left(\left|x\right|-\epsilon\right) & \left|x\right|>\epsilon.
\end{cases}
\]
Integrating using polar coordinates, we get 
\begin{eqnarray*}
\int f_{\epsilon} & = & n\omega_{n}\left[\int_{0}^{\epsilon}1\cdot r^{n-1}dr+\int_{\epsilon}^{\infty}h(r-\epsilon)r^{n-1}dr\right]\\
 & = & n\omega_{n}\left[\frac{\epsilon^{n}}{n}+\int_{0}^{\infty}h(r)(r+\epsilon)^{n-1}dr\right]\\
 & = & \omega_{n}\epsilon^{n}+n\omega_{n}\cdot\sum_{i=0}^{n-1}\int_{0}^{\infty}h(r)\cdot\binom{n-1}{i}r^{n-i-1}dr\cdot\epsilon^{i},
\end{eqnarray*}
 where $\omega_{n}=\vol\left(D\right)$ is the volume of the unit
ball. Comparing this with the definition of the $W_{i}$'s as 
\[
\int f_{\epsilon}=\sum_{i=0}^{n}\binom{n}{i}W_{i}(f)\epsilon^{i},
\]
 we see that for every $0\le i<n$ we have 
\[
W_{i}(f)=\frac{n\omega_{n}\binom{n-1}{i}\cdot\int_{0}^{\infty}h(r)\cdot r^{n-i-1}dr}{\binom{n}{i}}=\left(n-i\right)\omega_{n}\int_{0}^{\infty}h(r)\cdot r^{n-i-1}dr.
\]

Now we use Proposition \ref{prop:moments} with $k$ and $m$ replaced
with $n-m-1$ and $n-k-1$. We get
\[
\left[\frac{1}{\Gamma(n-k)}\int_{0}^{\infty}r^{n-k-1}h(r)dr\right]^{\frac{1}{n-k}}\le\left[\frac{1}{\Gamma(n-m)}\int_{0}^{\infty}r^{n-m-1}h(r)dr\right]^{\frac{1}{n-m}}
\]
 or 
\[
\left[\frac{W_{k}(f)}{\left(n-k\right)\omega_{n}\Gamma(n-k)}\right]^{\frac{1}{n-k}}\le\left[\frac{W_{m}(f)}{(n-m)\omega_{n}\Gamma(n-m)}\right]^{\frac{1}{n-m}}.
\]
For the function $g(x)=e^{-\left|x\right|}$ we know we have an equality
in Proposition \ref{prop:moments} , so 
\[
\left[\frac{W_{k}(g)}{\left(n-k\right)\omega_{n}\Gamma(n-k)}\right]^{\frac{1}{n-k}}=\left[\frac{W_{m}(g)}{(n-m)\omega_{n}\Gamma(n-m)}\right]^{\frac{1}{n-m}}.
\]
Dividing the equations, we get 
\[
\left(\frac{W_{k}(f)}{W_{k}(g)}\right)^{\frac{1}{n-k}}\le\left(\frac{W_{m}(f)}{W_{m}(g)}\right)^{\frac{1}{n-m}}.
\]

Finally, by the equality cases of Propositions \ref{prop:Alexandrov}
and \ref{prop:moments}, we get an equality if and only if $f$ is
rotation invariant and $h(r)=e^{-cr}$, which means that $f(x)=e^{-c\left|x\right|}$.
\end{proof}
In the case $k=0$, $m=1$, we immediately obtain a sharp isoperimetric
inequality (remember that, by definition, $S(f)=n\cdot W_{1}(f)$
):
\begin{prop}
\label{prop:lc-isoperimetric}For every $f\in\lco$ we have
\[
S(f)\ge\left(\int f\right)^{\frac{n-1}{n}}\cdot\frac{S(g)}{\left(\int g\right)^{\frac{n-1}{n}}}
\]
 with equality if and only if $f(x)=e^{-c\left|x\right|}$ for some
$c>0$. \end{prop}
\begin{rem}
\label{rem:counter-isoperimetric}In Proposition \ref{prop:lc-isoperimetric}
we made two assumptions about $f$: it must be log-concave, and it
must be geometric. Both assumptions are absolutely crucial, as we
will now see.

Define $f:\R^{2}\to[0,\infty)$ by $f(x)=a^{2}e^{-a\left|x\right|}.$
The function $f$ is log-concave, but not geometric unless $a=1$.
Strictly speaking, we only defined the quermassintegrals for geometric
functions, but from the proof of Theorem \ref{thm:lc-alexandrov}
we immediately see that $\int f_{\epsilon}$ is a polynomial in $\epsilon$,
and the coefficients are 
\[
\int f=W_{0}(f)=2\pi\cdot\int_{0}^{\infty}a^{2}e^{-ar}\cdot rdr=2\pi,
\]
 while 
\[
S(f)=2W_{1}(f)=2\pi\cdot\int_{0}^{\infty}a^{2}e^{-ar}=2\pi a.
\]
 By taking $a\to0$ we see that it is indeed impossible to get any
lower bound on $S(f)$ in terms of $\int f$. 

Similarly, for $a>2$ define $f:\R^{2}\to[0,\infty)$ by $f(x)=\left(1+\frac{\left|x\right|}{\sqrt{a^{2}-3a+2}}\right)^{-a}$.
The function $f$ is geometric and quasi-concave, but not log-concave.
Again using the same formulas we get
\begin{eqnarray*}
\int f & = & 2\pi\cdot\int_{0}^{\infty}r\left(1+\frac{r}{\sqrt{a^{2}-3a+2}}\right)^{-a}dr=2\pi\\
S(f) & = & 2\pi\cdot\int_{0}^{\infty}\left(1+\frac{r}{\sqrt{a^{2}-3a+2}}\right)^{-a}dr=2\pi\cdot\sqrt{\frac{a-2}{a-1}}.
\end{eqnarray*}
 Taking $a\to2^{+}$, we see that it is again impossible to bound
$S(f)$ from below using $\int f$. 
\end{rem}

\begin{rem}
We stated Theorem \ref{thm:lc-alexandrov} and Proposition \ref{prop:lc-isoperimetric}
for log-concave functions, but similar results can also be stated
for $\alpha$-concave functions, for every non-positive value of $\alpha$
(see \cite{brascamp_extensions_1976} for definitions) . In the class
of $\alpha$-concave functions, the extremal function will not be
$g(x)=e^{-\left|x\right|}$, but $g(x)=\left(1-\alpha\left|x\right|\right)^{1/\alpha}$.
Since we have not discussed $\alpha$-concave functions in this paper,
and since the generalized proofs are almost identical to the ones
we gave, we will not pursue this point any further. 
\end{rem}

\section{\label{sec:Rescaling}Rescalings and dilations}

In this last section, we will explore the notion of rescaling, discussed
in the introduction. We formally define: 
\begin{defn}
A rescaling of a function $f\in\qco$ is a function of the form $\alpha\circ f$,
where $\alpha:[0,1]\to[0,1]$ is an increasing bijection.
\end{defn}
It is easy to see that if $\tilde{f}=\alpha\circ f$ is a rescaling
of $f$, then 
\[
\KU_{t}(\tilde{f})=\KU_{\alpha^{-1}(t)}(f).
\]
Rescaling will be especially effective if the function $f$ satisfies
certain regularity assumptions. For concreteness, let us define:
\begin{defn}
A function $f\in\qco$ is called regular if 
\begin{enumerate}
\item $f$ is continuous.
\item $f(\lambda x)>f(x)$ for all $x\in\R^{n}$ and $0<\lambda<1$. 
\item $f(x)\to0$ as $\left|x\right|\to\infty$.
\end{enumerate}
\end{defn}
We will need the following technical lemma:
\begin{lem}
\label{lem:technical}Let $\Phi$ be a size functional, and let $f\in\qco$
be regular. Then the map $\phi_{f}:[0,1]\to[0,\infty]$ defined by
\[
\phi_{f}(t)=\Phi\left(\KU_{t}(f)\right)
\]
 is a decreasing bijection.\end{lem}
\begin{proof}
First we notice that for every $0<t\le1$, the set $\KU_{t}(f)$ is
compact: it is closed because $f$ is continuous, and bounded because
$f(x)\to0$ as $\left|x\right|\to\infty$. Also, for every $0\le t<1$,
the set $\KU_{t}(f)$ has non-empty interior, because it contains
an $\epsilon$-neighborhood of $0$. 

Now let us show that $\phi_{f}$ is strictly decreasing. Fix $0<s<t<1$.
Then $\KU_{t}(f)$ is compact, $\left\{ x:\ f(x)\le s\right\} $ is
closed and these two sets are disjoint. It follows that they are $\epsilon$-separated
for some $\epsilon>0$, i.e. 
\[
\left[\KU_{t}(f)+\epsilon B_{2}^{n}\right]\cap\left\{ x:\ f(x)\le s\right\} =\emptyset.
\]
 This implies that 
\[
\KU_{t}(f)+\epsilon B_{2}^{n}\subseteq\left\{ x:\ f(x)>s\right\} \subseteq\KU_{s}(f),
\]
 so $\phi_{f}(t)=\Phi\left(\KU_{t}(f)\right)<\Phi\left(\KU_{s}(f)\right)=\phi_{f}(s)$
like we wanted. 

We still need to check the end points of $[0,1]$. For $t=0$ we get
\[
\phi_{f}(0)=\Phi\left(\KU_{0}(f)\right)=\Phi\left(\R^{n}\right)=\infty,
\]
 but if $t>0$ then $\phi_{f}(t)<\infty$ because $\KU_{t}(f)$ is
compact. Similarly, from the definition of regularity we see that
$f(x)<f(0)=1$ for all $x\ne0$, so $\KU_{1}(f)=\left\{ 0\right\} $.
Hence we get 
\[
\phi_{f}(1)=\Phi\left(\KU_{1}(f)\right)=\Phi\left(\left\{ 0\right\} \right)=0,
\]
 but if $t<1$ then $\phi_{f}(t)>0$ since $\KU_{t}(f)$ has non-empty
interior. This completes the proof that $\phi_{f}$ is strictly decreasing,
hence injective.

Now we wish to prove that $\phi_{f}$ is continuous. To do so we will
need the observation that for every $0\le t\le1$
\[
\left\{ x:\ f(x)>t\right\} =\text{int}\left[\KU_{t}(f)\right],
\]
 where $\text{int}$ denotes the topological interior. Indeed, the
inclusion $\subseteq$ is obvious since the set $\left\{ x:\ f(x)>t\right\} $
is open. For the other inclusion, assume $x\in\text{int}\left[\KU_{t}(f)\right]$,
then $(1+\epsilon)x\in\KU_{t}(f)$ for small enough $\epsilon>0$.
This implies that 
\[
f(x)=f\left(\frac{1}{1+\epsilon}\cdot\left(1+\epsilon\right)x\right)>f\left((1+\epsilon)x\right)\ge t,
\]
 so we proved the claim.

Now continuity follows easily: from the left we have 
\[
\bigcap_{s<t}\KU_{s}(f)=\KU_{t}(f),
\]
 so by continuity of classic mixed volumes we get 
\[
\lim_{s\to t^{-}}\phi_{f}(t)=\lim_{s\to t^{-}}\Phi\left(\KU_{s}(f)\right)=\Phi\left(\KU_{t}(f)\right)=\phi_{f}(t).
\]
 Similarly, from the right, we get 
\[
\bigcup_{s>t}\KU_{s}(f)=\left\{ x:\ f(x)>t\right\} =\text{int}\left[\KU_{t}(f)\right],
\]
 and again by continuity of mixed volumes we get $\lim_{s\to t^{+}}\phi_{f}(s)=\phi_{f}(t)$. 

Since $\phi_{f}$ is continuous the image $\phi_{f}\left([0,1]\right)$
is connected, and since we already saw that $0,\infty\in\phi_{f}\left([0,1]\right)$
it follows that $\phi_{f}$ is onto. Hence our proof is complete.
\end{proof}
Using the lemma, we can achieve the goals promised in the introduction.
Specifically, we prove the following generalized Brunn-Minkowski inequality:
\begin{prop}
Assume $f,g\in\qco$ are regular, and fix a size functional $\Phi$
of degree $k$. Then one can rescale $f$ to a function $\tilde{f}$
in such a way that 
\[
\Phi(\widetilde{f}\oplus g)^{\frac{1}{k}}\ge\Phi(\widetilde{f})^{\frac{1}{k}}+\Phi(g)^{\frac{1}{k}}.
\]
\end{prop}
\begin{proof}
Using the notation of Lemma \ref{lem:technical}, define $\alpha:[0,1]\to[0,1]$
by $\alpha=\phi_{g}^{-1}\circ\phi_{f}$. By the lemma $\alpha$ is
an increasing bijection, so $\widetilde{f}=\alpha\circ f$ is a rescaling
of $f$. By direct calculation 
\begin{eqnarray*}
\Phi\left(\KU_{t}(\widetilde{f})\right) & = & \Phi\left(\KU_{\alpha^{-1}(t)}(f)\right)=\phi_{f}\left(\alpha^{-1}(t)\right)=\left[\phi_{f}\circ\phi_{f}^{-1}\circ\phi_{g}\right](t)\\
 & = & \phi_{g}(t)=\Phi\left(\KU_{t}(g)\right),
\end{eqnarray*}
 so $\widetilde{f}^{\Phi}=g^{\Phi}$. By Proposition \ref{prop:gen-BM}
we get 
\begin{eqnarray*}
\Phi(\widetilde{f}\oplus g)^{\frac{1}{k}} & \ge & \Phi(\widetilde{f}^{\Phi}\oplus g^{\Phi})^{\frac{1}{k}}=\Phi(2\odot g^{\Phi})^{\frac{1}{k}}=2\Phi(g^{\Phi})^{\frac{1}{k}}\\
 & = & \Phi(g^{\Phi})^{\frac{1}{k}}+\Phi(g^{\Phi})^{\frac{1}{k}}=\Phi(\widetilde{f}^{\Phi})^{\frac{1}{k}}+\Phi(g^{\Phi})^{\frac{1}{k}}\\
 & = & \Phi(\widetilde{f})^{\frac{1}{k}}+\Phi(g)^{\frac{1}{k}}.
\end{eqnarray*}

\end{proof}
Notice that we only needed to rescale one of the functions (in this
case $f$), but the exact rescaling depended on $g$. The same result
can be obtained by rescaling both $f$ and $g$, but in a universal
way -- the rescaling of $f$ will depend on $\Phi$ but not on $g$,
and vice versa. This is not hard to see -- just choose a fixed, ``universal'',
regular quasi-concave function $h$, and use the same technique we
used in the proof to rescale both $f$ and $g$ in such a way that
$\Phi\left(\KU_{t}(\widetilde{f})\right)=\Phi\left(\KU_{t}(\widetilde{g})\right)=\Phi\left(\KU_{t}(h)\right).$

As a second remark, note that we have an extra degree of freedom which
we have not used. We chose our rescaling $\alpha$ in such a way that
$\widetilde{f}^{\Phi}=g^{\Phi}$, but for any $c>0$ we could have
chosen $\alpha$ to satisfy that $\widetilde{f}^{\Phi}=c\odot g^{\Phi}$,
and the proof would have worked in exactly the same way. Using this
degree of freedom we may for example choose $\widetilde{f}$ to satisfy
$\Phi(\widetilde{f})=\Phi(f)$, or alternatively $\int\widetilde{f}=\int f$. 

In a similar way, one can obtain a version of the Alexandrov-Fenchel
inequality
\begin{prop}
Assume $f_{1},f_{2},\ldots,f_{m}\in\qco$ are regular functions, and
$A_{1},A_{2},\ldots,A_{n-m}\in\kn$ are compact bodies with non-empty
interior. Then one can rescale each $f_{i}$ to a function $\tilde{f}_{i}$
such that 
\[
V(\widetilde{f_{1}},\widetilde{f_{2}},\ldots,\widetilde{f_{m}},\o_{A_{1}},\ldots,\o_{A_{n-m}})^{m}\ge\prod_{i=1}^{m}V(\widetilde{f_{i}},\widetilde{f_{i}},\ldots,\widetilde{f_{i}},\o_{A_{1}},\ldots,\o_{A_{n-m}}).
\]
 \end{prop}
\begin{proof}
We will use Lemma \ref{lem:technical} again, this time with 
\[
\Phi(K)=V(K,K,\ldots,K,A_{1},A_{2},\ldots,A_{n-m}).
\]
 Fix some regular quasi-concave function $h$, and scale each $f_{i}$
using $\alpha=\phi_{h}^{-1}\circ\phi_{f}$. Like before we will have
$\widetilde{f_{i}}^{\Phi}=h^{\Phi}$ for all $i$. Thus, using Proposition
\ref{prop:AF} we get
\begin{eqnarray*}
V(\widetilde{f_{1}},\widetilde{f_{2}},\ldots,\widetilde{f_{m}},\o_{A_{1}},\ldots,\o_{A_{n-m}})^{m} & \ge & V(\widetilde{f_{1}}^{\Phi},\widetilde{f_{2}}^{\Phi},\ldots,\widetilde{f_{m}}^{\Phi},\o_{A_{1}},\ldots,\o_{A_{n-m}})^{m}\\
 & = & V(\widetilde{h}^{\Phi},\widetilde{h}^{\Phi},\ldots,\widetilde{h}^{\Phi},\o_{A_{1}},\ldots,\o_{A_{n-m}})^{m}\\
 & = & \prod_{i=1}^{m}V(\widetilde{h}^{\Phi},\widetilde{h}^{\Phi},\ldots,\widetilde{h}^{\Phi},\o_{A_{1}},\ldots,\o_{A_{n-m}})\\
 & = & \prod_{i=1}^{m}V(\widetilde{f_{i}}^{\Phi},\widetilde{f_{i}}^{\Phi},\ldots,\widetilde{f_{i}}^{\Phi},\o_{A_{1}},\ldots,\o_{A_{n-m}})\\
 & = & \prod_{i=1}^{m}V(\widetilde{f_{i}},\widetilde{f_{i}},\ldots,\widetilde{f_{i}},\o_{A_{1}},\ldots,\o_{A_{n-m}}).
\end{eqnarray*}

As a corollary, we immediately get\end{proof}
\begin{cor}
Assume $f_{1},f_{2},\ldots,f_{n}\in\qc$ are regular. Then it is possible
to rescale each $f_{i}$ to a function $\tilde{f_{i}}$ in such a
way that
\[
V(\tilde{f}_{1},\tilde{f_{2}},\ldots,\tilde{f_{n}})\ge\left(\prod_{i=1}^{n}\int\tilde{f}_{i}\right)^{\frac{1}{n}}.
\]

\end{cor}
The idea of rescaling is simple and powerful, but unfortunately it
does not apply to all functions. For example, if $K$ is a convex
body then the indicator $\o_{K}$ is definitely not regular, so we
cannot use the above propositions. To conclude this paper we will
now describe another procedure, similar to rescaling, which does not
assume regularity. The idea is to take the function $f$, and dilate
each level set $\KU_{t}(f)$ to the required ``size''. In other
words, given some size functional $\Phi$, we want to construct a
function $\tilde{f}$ such that 
\[
\KU_{t}(\tilde{f})=A(t)\cdot\KU_{t}(f),
\]
 and $\Phi\left(\KU_{t}(\tilde{f})\right)=\phi(t)$ for some prescribed
$\phi$. 

The problem is that for general quasi-concave functions $f$ and general
laws $\phi$, such a $\tilde{f}$ may not exist. The following proposition
gives one case where the existence of $\tilde{f}$ is guaranteed:
\begin{prop}
Fix a size functional $\Phi:\kn\to[0,\infty]$ and a geometric log-concave
function $f\in\lco$. Define $M(x)=e^{-\left|x\right|}$. Then it
is possible to construct a function $\tilde{f}$ such that $\KU_{t}(\tilde{f})$
is always homothetic to $\KU_{t}(f)$, and 
\[
\phi_{\widetilde{f}}(t):=\Phi\left(\KU_{t}(\widetilde{f})\right)=\Phi\left(\KU_{t}(M)\right)=\phi_{M}(t)
\]
 for all $t$. The function $\tilde{f}$ will be called a dilation
of $f$. \end{prop}
\begin{proof}
Assume $\Phi$ is of degree $m$. The idea is to construct $\widetilde{f}$
such that 
\[
\KU_{t}(\widetilde{f})=\left(\frac{\Phi\left(\KU_{t}(M)\right)}{\Phi\left(\KU_{t}(f)\right)}\right)^{\frac{1}{m}}\cdot\KU_{t}(f).
\]
 It is obvious that for such an $\widetilde{f}$ we will have $\phi_{\widetilde{f}}=\phi_{M}$.
The only thing we need to prove is that such an $\widetilde{f}$ really
exists, that is that the family of convex bodies $\left\{ \KU_{t}(\widetilde{f})\right\} $
is really the level sets of some function. This will follow easily
once we prove that these level sets are monotone: if $t\le s$ then
$\KU_{t}(\widetilde{f})\supseteq\KU_{s}(\widetilde{f})$.

Fix $0<t\le s\le1$. By direct computation, 
\[
\KU_{t}(M)=\left\{ x:\ e^{-\left|x\right|}\ge t\right\} =\left\{ x:\ \left|x\right|<\log\left(\frac{1}{t}\right)\right\} =\log\left(\frac{1}{t}\right)\cdot D.
\]
 Define 
\[
\lambda=\left(\frac{\Phi\left(\KU_{s}(M)\right)}{\Phi\left(\KU_{t}(M)\right)}\right)^{\frac{1}{m}}=\frac{\log\frac{1}{s}\cdot\Phi(D)^{\frac{1}{m}}}{\log\frac{1}{t}\cdot\Phi(D)^{\frac{1}{m}}}=\frac{\log\frac{1}{s}}{\log\frac{1}{t}}.
\]
 Notice that for every $x\in\KU_{t}(f)$ we have 
\[
f(\lambda x)=f(\lambda x+(1-\lambda)0)\ge f(x)^{\lambda}\cdot1^{1-\lambda}\ge t^{\lambda}=s,
\]
 so $\lambda x\in\KU_{s}(f)$. It follows that $\lambda\KU_{t}(f)\subseteq\KU_{s}(f)$,
so 
\[
\Phi\left(\KU_{s}(f)\right)\ge\lambda^{m}\cdot\Phi\left(\KU_{t}(f)\right)=\frac{\Phi\left(\KU_{s}(M)\right)}{\Phi\left(\KU_{t}(M)\right)}\cdot\Phi\left(\KU_{t}(f)\right),
\]
or 
\[
\frac{\Phi\left(\KU_{t}(M)\right)}{\Phi\left(\KU_{t}(f)\right)}\ge\frac{\Phi\left(\KU_{s}(M)\right)}{\Phi\left(\KU_{s}(f)\right)}.
\]
Hence we definitely have 
\[
\KU_{t}(\widetilde{f})=\left(\frac{\Phi\left(\KU_{t}(M)\right)}{\Phi\left(\KU_{t}(f)\right)}\right)^{\frac{1}{m}}\cdot\KU_{t}(f)\supseteq\left(\frac{\Phi\left(\KU_{s}(M)\right)}{\Phi\left(\KU_{s}(f)\right)}\right)^{\frac{1}{m}}\cdot\KU_{s}(f)=\KU_{s}(\widetilde{f}),
\]
and the proof is complete.
\end{proof}
We see that if $f\in\lco$, then $\widetilde{f}\in\qco$. However,
the function $\widetilde{f}$ may fail to be log-concave, as the next
example shows.
\begin{example}
Define $f(x,y)=e^{-\left(\left|x\right|+y^{2}\right)}\in\text{LC}_{0}(\R^{2})$,
and choose $\Phi=\vol$. Notice that 
\[
\left|\left\{ \left|x\right|+y^{2}\le c\right\} \right|=\int_{y=-\sqrt{c}}^{\sqrt{c}}\int_{x=y^{2}-c}^{c-y^{2}}dxdy=\frac{8}{3}c^{\frac{3}{2}},
\]
 so 
\[
\left|\KU_{t}(f)\right|=\left|\left\{ e^{-\left(\left|x\right|+y^{2}\right)}\ge t\right\} \right|=\left|\left\{ \left|x\right|+y^{2}\le\log\frac{1}{t}\right\} \right|=\frac{8}{3}\left(\log\frac{1}{t}\right)^{\frac{3}{2}},
\]
 while 
\[
\left|\KU_{t}(M)\right|=\left|\left\{ e^{-\sqrt{x^{2}+y^{2}}}\ge t\right\} \right|=\left|\left\{ \sqrt{x^{2}+y^{2}}\le\log\frac{1}{t}\right\} \right|=\pi\left(\log\frac{1}{t}\right)^{2}.
\]
 Therefore in this case we get 
\begin{eqnarray*}
\KU_{t}(\widetilde{f}) & = & \left(\frac{\left|\KU_{t}(M)\right|}{\left|\KU_{t}(f)\right|}\right)^{\frac{1}{2}}\cdot\KU_{t}(f)=\left(\frac{\pi\log^{2}\frac{1}{t}}{\frac{8}{3}\log^{\frac{3}{2}}\frac{1}{t}}\right)^{\frac{1}{2}}\KU_{t}(f)=C\cdot\log^{\frac{1}{4}}\frac{1}{t}\cdot K_{t}(f)=\\
 & = & \left\{ (x,y):\ \frac{\left|x\right|}{C\log^{\frac{1}{4}}\frac{1}{t}}+\frac{y^{2}}{C^{2}\log^{\frac{1}{2}}\frac{1}{t}}\le\log\frac{1}{t}\right\} ,
\end{eqnarray*}
 for some explicit constant $C$. In other words, $\widetilde{f}(x,y)=t$,
where $t\in(0,1]$ is the unique solution to the equation 
\[
\frac{\left|x\right|}{C\log^{\frac{1}{4}}\frac{1}{t}}+\frac{y^{2}}{C^{2}\log^{\frac{1}{2}}\frac{1}{t}}=\log\frac{1}{t}.
\]

In general this is difficult to solve explicitly, but for $y=0$ we
get that $\widetilde{f}(x,0)$ is the solution of 
\[
\frac{\left|x\right|}{C\log^{\frac{1}{4}}\frac{1}{t}}=\log\frac{1}{t},
\]
 so 
\[
\widetilde{f}(x,0)=e^{-\widetilde{C}\left|x\right|^{\frac{4}{5}}}.
\]
 This is enough to conclude that $\widetilde{f}$ is not a log-concave
function, even though $f$ is.
\end{example}
Using this proposition, we can prove our main propositions again,
with rescalings replaced by dilations. As the proofs are almost identical,
we will only state the results:
\begin{prop}
Assume $f,g\in\lco$, and fix a size functional $\Phi$ of degree
$k$. Then one can dilate $f$ and $g$ to functions $\tilde{f}$,
$\tilde{g}$ in such a way that 
\[
\Phi(\widetilde{f}\oplus\widetilde{g})^{\frac{1}{k}}\ge\Phi(\widetilde{f})^{\frac{1}{k}}+\Phi(\widetilde{g})^{\frac{1}{k}}.
\]

\end{prop}

\begin{prop}
Assume $f_{1},f_{2},\ldots,f_{m}\in\lco$ , and $A_{1},A_{2},\ldots,A_{n-m}\in\kn$
are compact bodies with non-empty interior. Then one can dilate each
$f_{i}$ to a function $\tilde{f}_{i}$ such that 
\[
V(\widetilde{f_{1}},\widetilde{f_{2}},\ldots,\widetilde{f_{m}},\o_{A_{1}},\ldots,\o_{A_{n-m}})^{m}\ge\prod_{i=1}^{m}V(\widetilde{f_{i}},\widetilde{f_{i}},\ldots,\widetilde{f_{i}},\o_{A_{1}},\ldots,\o_{A_{n-m}}).
\]

\end{prop}

\subsubsection*{Acknowledgment}

The authors would like to thank Prof. R. Schneider for his detailed
comments regarding the written text.

\bibliographystyle{plain}
\bibliography{hyperbolic-paper-new}

\end{document}